\documentclass[leqno,10.5pt]{article} %leqno is the option to put formulanumbers on the left side
\usepackage{amsmath, amssymb}  
\usepackage{amsthm} %theorem environment option
\usepackage{amscd}   
\usepackage{amsmath, amssymb}    
\usepackage{amsthm} %theorem environment option 
\usepackage{ascmac}
\usepackage{amscd}       
\usepackage{color}

\theoremstyle{plain} %text of this environment is typesetted in italics
\newtheorem{theorem}{\indent\sc Theorem}[section]
\newtheorem{lemma}[theorem]{\indent\sc Lemma}
\newtheorem{corollary}[theorem]{\indent\sc Corollary}
\newtheorem{proposition}[theorem]{\indent\sc Proposition}

\theoremstyle{definition} %text of this environment is typesetted in roman letters
\newtheorem{definition}[theorem]{\indent\sc Definition}

\newtheorem{remark}[theorem]{\indent\sc Remark}
\newtheorem{example}[theorem]{\indent\sc Example}

\newcommand{\C}{\mathbb{C}}
\newcommand{\R}{\mathbb{R}}
\newcommand{\Q}{\mathbb{Q}}

\newcommand{\Z}{\mathbb{Z}}

\newcommand{\N}{\mathbb{N}}

\def\2{I\hspace{-.1em}I}

\title{Hermite's approach to Abelian integrals revisited} %via Pad\'{e}-type approximations} %: \\ applications to Lauricella hypergeometric values}
\author{\textsc{Makoto Kawashima}}
\date{ }   

\begin{document}

\maketitle

\rightline{\textit{Dedicated to the Memory of Professor Marc Huttner}}
 
%\tableofcontents

%\footnote{ %key words and phrases

%\footnote{ %acknowledgment of support etc. if any
%$^{*}$Thanks.
%}
%%%%%%%%%%%%%%%%%%%%%%%%%%%%%%%%%%%%%%%%
%\begin{abstract}

%\end{abstract}
%\textit{Key words}:~Pad\'e approximants, Rodrigues formula, linear independence, Gauss hypergeometric functions.
%%%%%%%%

\begin{abstract} 
In this article, we establish a new linear independence criterion for the values of certain {\it Lauricella hypergeometric series} $F_D$ with rational parameters, in both the complex and 
$p$-adic settings, over an algebraic number field. This result generalizes a theorem of C.~Hermite \cite{Hermite} on the linear independence of certain Abelian integrals.
Our proof relies on explicit Pad\'{e}-type approximations to solutions of a reducible Jordan-Pochhammer differential equation, which extends the Pad\'{e} approximations for certain Abelian integrals in \cite{Hermite}.
The main novelty of our approach lies in the proof of the non-vanishing of the determinants associated with these Pad\'{e}-type approximants.
\end{abstract}


\section{Introduction}
In this paper we extend Hermite's construction of Pad\'{e} approximants for Abelian integrals to a broad class of Lauricella hypergeometric series, and derive an explicit criterion (with an effective measure) for the linear independence of their values over algebraic number fields - simultaneously in the complex and 
$p$-adic contexts.

Pad\'{e} approximation, originating in the works of C.~Hermite in his study of the transcendence of~$e$~\cite{Hermite e} and H.~Pad\'{e}~\cite{Pade1, Pade2}, has long played a central role in Diophantine approximation and transcendental number theory.  
In arithmetic applications, one typically constructs explicit systems of Pad\'{e} approximations to certain functions, often by linear algebraic arguments combined with bounds derived from Siegel's lemma via Dirichlet's box principle.  
However, these general constructions are not always sufficient for arithmetic purposes, such as establishing linear independence results.  
In such cases, it becomes essential to construct explicit Pad\'{e} approximations providing analytic estimates sharp enough for the intended application--a task that can usually be carried out only for special classes of functions.

\smallskip

Over the complex field, Hermite~\cite{H} established a criterion for the $\mathbb{Q}$-linear independence of the values of certain Abelian integrals associated with the Jordan-Pochhammer differential equation, by explicitly constructing Pad\'{e} approximations of such integrals.  
In his work, all parameters~$s_i$ determining the exponents of the Jordan-Pochhammer equation (see equation~\eqref{second}) are assumed to be of the form~$1/k$ for positive integers~$k$.  
Following Hermite's approach, G.~Rhin and P.~Toffin~\cite{R-T} proved a linear independence criterion for distinct logarithmic values over imaginary quadratic fields by constructing explicit Pad\'{e}-type approximants of logarithmic functions, corresponding to the case~$s_i = 0$.  
Later, M.~Huttner~\cite{H} refined Hermite's method and, through a detailed analysis of Hermite's approximants, obtained sharper measures of linear independence than those of Hermite's original results (see also \cite{H0}).  
In~\cite{Kaw}, the present author established a linear independence criterion for the values of Gauss' hypergeometric functions with varying parameters, both in the complex and in the~$p$-adic settings, for a special case where all~$s_i = -1/m$ with~$m$ the number of parameters.  
Related Diophantine properties of the values of solutions to Jordan-Pochhammer equations were investigated by G.~V.~and D.~V.~Chudnovsky in~\cite[Section~2]{ch brothers II}.

\smallskip

In the present work, we extend Hermite's criterion to the case where the parameters $s_i$ are arbitrary rational numbers whose sum does not lie in $\mathbb{Z}_{<-1}$.  
This leads to a new linear independence criterion for the corresponding Abelian integrals over algebraic number fields, valid simultaneously in the complex and $p$-adic settings.  
Our approach relies on a detailed study of Pad\'{e}-type approximants (beyond the classical Pad\'{e} approximants) and their algebraic and analytic structures, following the framework developed in~\cite{Kaw}.

To investigate the arithmetic nature of the values of holonomic Laurent series, it is crucial to construct their Pad\'{e}-type approximants and to elucidate their analytic and algebraic properties.  
In~\cite{Kaw}, explicit constructions were given for holonomic series on which a first-order differential operator with polynomial coefficients acts so that the resulting function becomes a polynomial.  
In this paper, we establish more general sufficient conditions ensuring the non-vanishing of determinants associated with such Pad\'{e}-type approximants.  
This analysis is motivated by the goal of extending the method to broader classes of holonomic Laurent series and by the desire to understand the arithmetic implications of non-vanishing--a property that plays a central role in Siegel's method~\cite{Siegel1}.

A key technical ingredient is the introduction of the {\it formal $f$-integration map} $\varphi_f$ (see equation \eqref{varphi f}), which enables the explicit--yet purely formal--construction of Pad\'{e}-type approximants and the proof of their principal properties.  
This concept appears in various forms in earlier works of S.~David, N.~Hirata-Kohno, and the author~\cite{DHK2,DHK3,DHK4,DHK5}, as well as in~\cite{KP,KP2} by A.~Po\"{e}ls and the author.

To apply Siegel's method~\cite{Siegel1} in proving our main theorem, it is essential to establish the non-vanishing of the determinant formed by these Pad\'{e}-type approximants.  
In previous works such as~\cite{DHK2,DHK3,DHK4,KP,Kaw}, this step required explicit computation of the determinant, which is often a intricate and technically demanding task.  
Here, we develop a new approach based on the study of the kernel of the formal $f$-integration map, as introduced in~\cite{KP2} and~\cite{DHK5}.  
This method allows us to prove the non-vanishing property in a simpler and more conceptual way, without computing the explicit determinant values.  
We emphasize that the non-vanishing condition is governed entirely by the first-order differential operators involved--specifically, by the coefficients of operator (see Theorem~$\ref{nonvanish Delta}$).

\subsection{Main result} \label{subsection theorem}
Our main result concerns a linear independence criterion for the values of the Lauricella hypergeometric series over an algebraic number field, in both the complex and $p$-adic settings.
We begin by recalling the definition of the $m$-variable Lauricella hypergeometric series with parameters $\alpha, \beta_1,\ldots,\beta_m, \gamma$, defined by
\[
F^{(m)}_{D} \!\left(\alpha,\beta_1,\ldots,\beta_m,\gamma; z_1,\ldots,z_m\right)
=\sum_{k_1,\ldots,k_m=0}^{\infty}
\dfrac{(\alpha)_{k_1+\cdots+k_m}(\beta_1)_{k_1}\cdots(\beta_m)_{k_m}}
{(\gamma)_{k_1+\cdots+k_m}\,k_1!\cdots k_m!}\,
z_1^{k_1}\cdots z_m^{k_m},
\]
where $(a)_k$ denotes the Pochhammer symbol, given by $(a)_0=1$ and $(a)_k=a(a+1)\cdots(a+k-1)$ for $k\ge 1$.

\medskip

To state our main result, we first fix notation.
Let $K$ be an algebraic number field and denote by $\mathfrak{M}_K$ the set of places of $K$, and for $v\in \mathfrak{M}_K$, let $K_v$ denote the completion of $K$ at $v$. 
We normalize the absolute value $|\cdot|_v$ by
\[
|p|_v=p^{-\tfrac{[K_v:\Q_p]}{[K:\Q]}} \quad \text{if } v\mid p,
\qquad
|x|_v=|\iota_v(x)|^{\tfrac{[K_v:\R]}{[K:\Q]}} \quad \text{if } v\mid \infty,
\]
where $p$ is a rational prime and $\iota_v:K\hookrightarrow \C$ the embedding associated to $v$. 
On $K_v^n$, the norm $|\cdot|_v$ is taken to be the supremum norm. 

For $\boldsymbol{\beta}=(\beta_1,\ldots,\beta_m)\in K^m$, we define the logarithmic $v$-adic height and the global logarithmic height by
\begin{align*}
{\mathrm{h}}_v(\boldsymbol{\beta})&=\log \max\{1, |\beta_1|_v,\ldots,|\beta_m|_v\}, \\
{\mathrm{h}}(\boldsymbol{\beta})&=\sum_{v\in \mathfrak{M}_K}{\mathrm{h}}_v(\boldsymbol{\beta}).
\end{align*}
We also define the denominator of $\boldsymbol{\beta}$ as
\[
{\rm{den}}(\boldsymbol{\beta})=\min\{n\in \Z\mid n>0 \ \text{such that} \ n\beta_i \ \text{are algebraic integers for all } 1\le i \le m\}.
\]

For a rational number $\alpha$, we put
\[
\mu(\alpha)={\rm{den}}(\alpha)\prod_{\substack{q:\,\text{prime} \\ q\mid {\rm{den}}(\alpha)}} q^{\tfrac{1}{q-1}}.
\]

\medskip

Now we are ready to state our main result. 
Let $m\ge 2$ be a fixed positive integer. Fix an algebraic number field $K$. 
Consider polynomials $a(z), b(z)\in K[z]$, where $a(z)$ is assumed to be monic of degree $m$ and $b(z)$ satisfies $\deg b\le m-1$. 
We denote the derivative of $a(z)$ by $a'(z)$, and by $b_{m-1}$ the coefficient of $z^{m-1}$ in $b(z)$. Note that $b_{m-1}$ can be $0$.
Assume that $a(z)$ decomposes over $K$, and let $\alpha_1,\ldots,\alpha_m\in K$ denote its roots, counted with multiplicity.
We impose the following conditions:
\begin{align}
&\alpha_1,\ldots,\alpha_{m} \ \text{are pairwise distinct}, \label{first} \\  
&s_i:=\dfrac{b(\alpha_i)}{a'(\alpha_i)}\in \Q \setminus \Z_{\le -1} 
\quad \text{for all } 1\le i \le m, \label{second} \\ 
& b_{m-1}\notin \Z_{< -1}. \label{third}
\end{align}
For $v\in \mathfrak{M}_K$ and $\beta\in K\setminus\{0\}$, we introduce the quantity
\begin{align} \label{V}
V_v(\beta)=& \, m\,{\rm{h}}_{v}(\beta)-(m-1){\rm{h}}(\beta)-\sum_{i=1}^m{\rm{h}}(\alpha_i) -m\left(
{\rm{h}}(\boldsymbol{\alpha})
+\sum_{i=1}^m\log \mu(s_i)+\log 4\right)\\
&-(m-1)\dfrac{{\rm{den}}(b_{m-1})}{\varphi({\rm{den}}(b_{m-1}))}\!
\sum_{\substack{1\le j \le {\rm{den}}(b_{m-1}) \\ (j,{\rm{den}}(b_{m-1}))=1}}
\dfrac{1}{j}, \nonumber
\end{align}
where $\boldsymbol{\alpha}=(\alpha_1,\ldots,\alpha_m)\in K^m$ and $\varphi$ denotes Euler's totient function.

\medskip

\begin{theorem} \label{main}
Retain the above notation and assumptions \eqref{first}, \eqref{second} and \eqref{third}.
For $0\le j \le m-2$, we define formal Laurent series $f_j(z)$ by
\[
\prod_{i=1}^m\left(1-\frac{\alpha_i}{z}\right)^{s_i}\cdot 
\frac{1}{z^{j+1}}\cdot 
F^{(m)}_D\!\left(
b_{m-1}+j+1,\,
1+s_1,\ldots,1+s_m;\,
b_{m-1}+j+2;\,
\frac{\alpha_1}{z},\ldots, \frac{\alpha_m}{z}
\right).
\]
Let $v_0\in \mathfrak{M}_K$. We assume $V_{v_0}(\beta)>0${\footnote{The positivity condition $V_{v_0}(\beta)>0$ roughly means that $\beta$ is arithmetically large enough at $v_0$ compared with the heights of the parameters. This guarantees both the convergence of the involved Laurent series at $z=\beta$ and the arithmetic control required in the application of Siegel's method.
}}. 
Then each series $f_j(z)$ converges at $z=\beta$ in $K_{v_0}$ and the $m$ elements of $K_{v_0}$$:$
$$1,f_0(\beta),f_1(\beta),\ldots,f_{m-2}(\beta),$$
are linearly independent over $K$.
\end{theorem}
We will give a proof of Theorem~$\ref{main}$ together with a linear independence measure (see Theorem~\ref{main+measure}).
\begin{remark} \label{fj are G functions}
We denote $L$ by the differential operator of order $1$ with polynomial coefficients $-a(z)\tfrac{d}{dz}+b(z)$.
A result due to S.~Fischler and T.~Rivoal \cite[Proposition~$3$ (ii)]{F-R} shows that $L$ is a $G$-operator (see the definition of $G$-operator \cite[IV]{An1}) under the assumptions \eqref{first} and \eqref{second}.
We observe that the Laurent series $f_j$ satisfies $L\cdot f_j\in K[z]$ with degree $m-j-2$ (see Lemma~\ref{formal solutions1}).
In particular, $f_j$ is a solution of a reducible Jordan-Pochhammer equation $\left(\tfrac{d}{dz}\right)^{m-1}L$ (see Example $\ref{reduce JP}$).
Combining these results, by a well-known theorem of Y.~Andr\'{e}, G.~V. \& D.~V.~Chudnovsky and N.~Katz (refer \cite[Th\'{e}or\`{e}me $3.5$]{An2}), we conclude $f_j$ is a $G$-function in the sense of C.~F.~Siegel \cite{Siegel1}.
\end{remark}

\medskip

\noindent\textbf{Outline of the article.} Section~\ref{Pade} is based on the results given in~\cite{Kaw}.  
We begin by introducing the Pad\'{e}-type approximants of Laurent series.  
In Subsection~\ref{formal map}, we introduce the {\it formal $f$-integration map} associated with a Laurent series $f$, which plays a central role throughout this paper, and we describe its fundamental properties in case of $f$ is holonomic.  
Subsection~\ref{Pade subsection} provides an overview of Pad\'{e}-type approximants and Pad\'{e}-type approximation for Laurent series that become polynomials under the action of a first-order differential operator with polynomial coefficients.
Section~\ref{Pade linear independence} formulates, in terms of the coefficients of the differential operator, sufficient conditions ensuring the non-vanishing of the determinant formed by the Pad\'{e}-type approximants constructed in Section~\ref{Pade} (see Theorem~\ref{nonvanish Delta}).  
This part constitutes the main novel contribution of the present work.
In Section~\ref{Estimate}, we treat the case where the differential operator considered in Section~\ref{Pade linear independence} is a $G$-operator.  
We give explicit expressions for the corresponding Laurent series and establish estimates for their Pad\'{e}-type approximants and Pad\'{e} approximations with respect to both Archimedean and non-Archimedean valuations.
Section~\ref{proof} is devoted to the proof of our main theorem, together with a quantitative measure of linear independence.  
Finally, Section~\ref{Appendix} serves as an appendix, summarizing some basic facts on the Jordan-Pochhammer equation.

\section{Pad\'{e}-type approximants of Laurent series} \label{Pade}
Throughout this section, we fix a field $K$ of characteristic $0$.
{{We denote the formal power series ring in the variable $1/z$ with coefficients $K$ by $K[[1/z]]$ and the field of fractions by $K((1/z))$. We say that an element of $K((1/z))$ is a formal Laurent series.}}
We define the order function at $z=\infty$ by
$${\rm{ord}}_{\infty}:K((1/z)) \longrightarrow \Z\cup \{\infty\}; \sum_{k} \dfrac{a_k}{z^k} \mapsto \min\{k\in\Z\cup \{\infty\} \mid a_k\neq 0\}.$$
{{Note that, for $f\in K((1/z))$, ${\rm{ord}}_{\infty} \, f=\infty$ if and only if $f=0$.}}
We recall without proof the following elementary fact:
\begin{lemma} \label{pade}
Let $m$ be a nonnegative integer, $f_1(z),\ldots,f_m(z)\in (1/z)\cdot K[[1/z]]$ and $\boldsymbol{n}=(n_1,\ldots,n_m)\in \N^{m}$.
Put $N=\sum_{j=1}^mn_j$.
For a nonnegative integer $M$ with $M\ge N$, there exist polynomials $(P,Q_{1},\ldots,Q_m)\in K[z]^{m+1}\setminus\{\boldsymbol{0}\}$ satisfying the following conditions:

\medskip

$(i)$ \ ${\rm{deg}}P\le M$,

\medskip

$(ii)$ \ ${\rm{ord}}_{\infty} \left(P(z)f_j(z)-Q_j(z)\right)\ge n_j+1 \ \ \text{for} \ \ 1\le j \le m$.
\end{lemma}
\begin{definition}
We say that a vector of polynomials $(P,Q_{1},\ldots,Q_m) \in K[z]^{m+1}$ satisfying properties $(i)$ and $(ii)$ is a weight $\boldsymbol{n}$, degree $M$ Pad\'{e}-type approximant of  $(f_1,\ldots,f_m)$.
For such approximants $(P,Q_{1},\ldots,Q_m)$ of $(f_1,\ldots,f_m)$, we call the formal Laurent series $(P(z)f_j(z)-Q_{j}(z))_{1\le j \le m}$, that is to say remainders, as weight $\boldsymbol{n}$ degree $M$ Pad\'{e}-type approximations of $(f_1,\ldots,f_m)$.
\end{definition}

\subsection{Formal $f$-integration map} \label{formal map}
Let $f(z)=\sum_{k=0}^{\infty} f_k/z^{k+1}\in (1/z)\cdot K[[1/z]]$.
We define a {{$K$-linear map}} $\varphi_f\in {\rm{Hom}}_K(K[t],K)$ by
\begin{align} \label{varphi f}
\varphi_f:K[t]\longrightarrow K; \ \ \ t^k\mapsto f_k \ \ \ (k\ge0) .
\end{align}

The above linear map extends naturally to a $K[z]$-{{linear map}} $\varphi_f: K[z,t]\rightarrow K[z]$, and then to a $K[z][[1/z]]$-{{linear map}} $\varphi_f: K[z,t][[1/z]]\rightarrow K[z][[1/z]]$. With this notation, the formal Laurent series $f(z)$ satisfies the following crucial identities {{(see \cite[$(6.2)$ page 60 and $(5.7)$ page 52]{N-S})}}:
\begin{align*}
&f(z)=\varphi_f \left(\dfrac{1}{z-t}\right),\ \ \ P(z)f(z)-\varphi_f\left(\dfrac{P(z)-P(t)}{z-t}\right)\in (1/z)\cdot K[[1/z]] \ \ \text{for any} \ \ P(z)\in K[z].
\end{align*}
We recall a criterion, based on the morphism $\varphi_f$, for determining when the given polynomials form Pad\'{e} approximants.

\begin{lemma}[{\it confer} {\rm{\cite[Lemma~2.3]{Kaw}}}] \label{equivalence Pade}
Let $m, M$ be positive integers, $f_1(z),\ldots,f_m(z)\in (1/z)\cdot K[[1/z]]$.
Let $\boldsymbol{n}=(n_1,\ldots,n_m)\in \N^m$ with $\sum_{j=1}^m n_j \le M$. 
Let $P(z)\in K[z]$ be a non-zero polynomial with  $\deg P\le M$, and put $Q_j(z)=\varphi_{f_j}\left((P(z)-P(t))/(z-t)\right)\in K[z]$ for $1\le j \le m$. The following assertions are equivalent.
   
\medskip
      
$(i)$ The vector of polynomials $(P,Q_1,\ldots,Q_m)$ is a weight $\boldsymbol{n}$ Pad\'{e}-type approximant of $(f_1,\ldots,f_m)$.

\medskip

$(ii)$ We have $t^kP(t)\in {\rm{ker}}\,\varphi_{f_j}$ for any pair of integers $(j,k)$ with $1\le j \le m$ and $0\le k \le n_j-1$.
   
\end{lemma}

\subsection{Holonomic series and the kernel of $f$-integration map}
Lemma~$\ref{equivalence Pade}$ implies that the study of the kernel of the formal $f$-integration map is essential for constructing Pad\'{e} approximants of Laurent series. 
This aspect has been studied in \cite{Kaw} for holonomic Laurent series. 
In this subsection we recall a result from \cite[Corollary~$2.6$]{Kaw}.

\medskip

We denote the action of a differential operator $L$ on a function $f$, such as a polynomial or a Laurent series, by $L\cdot f$.
Consider the map 
\begin{align}
\iota :K(z)[\tfrac{d}{dz}]\longrightarrow K(t)[\tfrac{d}{dt}]; \ \ \sum_j P_j(z)\tfrac{d^j}{dz^j} \mapsto \sum_j (-1)^j \tfrac{d^j}{dt^j} P_j(t). \label{iota diff}
\end{align}
{{Note that for $L\in K(z)[\tfrac{d}{dz}]$, $\iota(L)$ is called the {\emph {formal adjoint}} of $L$ and relates to the dual of differential module $K(z)[\tfrac{d}{dz}]/K(z)[\tfrac{d}{dz}]L$ (see \cite[III Exercises $3)$]{An1}).}}
For $L\in K(z)[\tfrac{d}{dz}]$, we denote $\iota(L)$ by $L^{*}$. 
Notice that we have $(L_1L_2)^{*}=L^*_2L^{*}_1$ for any $L_1,L_2\in K(z)[\tfrac{d}{dz}]$.  

\begin{proposition} $(${\rm{\cite[Corollary~$2.6$]{Kaw}}}$)$ \label{inc} 
Let $f(z)\in (1/z)\cdot K[[1/z]]$ and $L\in K[z,\tfrac{d}{dz}]$. Then the following are equivalent:

\medskip

$(i)$ $L\cdot f(z)\in K[z]$.

\medskip

$(ii)$ $L^{*}\cdot K[t]\subseteq {\rm{ker}}\,\varphi_f$.
\end{proposition}

\subsection{Order $1$ differential operator and Pad\'{e}-type approximants} \label{Pade subsection}
Let $K$ be a field of characteristic $0$ and $a(z),b(z)\in K[z]$. Put ${\rm{deg}}\,a=u, {\rm{deg}}\,b=v$, $w=\max\{u-2,v-1\}$ and
\begin{align} \label{a,b}
a(z)=\sum_{i=0}^u a_iz^i, \ \ b(z)=\sum_{j=0}^{v}b_jz^j.
\end{align} 
We now assume 
\begin{align} 
& \ \ \ \ \ \ \ \ \ \ \ \ \ \ \ \ \ \ \ \ \ \ \ \ \ \ \ \ \  w\ge0,\label{w}\\
&a_u(k+u)+b_{v} \neq 0 \ \ \ \text{for all} \ \  k\ge0 \ \ \text{if} \ \ u-1=v. \label{assump a,b}
\end{align}
Define the order $1$ differential operator with polynomial coefficients
$$L=-a(z)\dfrac{d}{dz}+b(z)\in K\left[z,\tfrac{d}{dz}\right].$$
We consider the Laurent series $f(z)\in (1/z)\cdot K[[1/z]]$ such that $L\cdot f(z)\in K[z]$ and construct the Pad\'{e} approximation of $f(z)$.
The following lemma is proven in \cite[Lemma~$4.1$]{Kaw}. For the sake of completion, we recall the proof in the present article. 
\begin{lemma} \label{solutions}
We keep the notation above.
Then there exist $f_0(z),\ldots,f_{w}(z)\in (1/z)\cdot K[[1/z]]$ that are linearly independent over $K$ and satisfy $L\cdot f_j(z)\in K[z]$ with at most degree $w$ for $0\le j \le w$. 
\end{lemma} 
\begin{proof}
Let $f(z)=\sum_{k=0}^{\infty}f_k/z^{k+1}\in (1/z)\cdot K[[1/z]]$ be a Laurent series. 
There exists a polynomial $A(z)\in K[z]$ that depends on the operator $L$ and $f$ with ${\rm{deg}}\,A\le w$ and satisfying
\begin{align} \label{Df}
L\cdot f(z)=A(z)
+\sum_{k=0}^{\infty}\dfrac{\sum_{i=0}^ua_i(k+i)f_{k+i-1}+\sum_{j=0}^vb_jf_{k+j}}{z^{k+1}}.
\end{align}
Put $$\sum_{i=0}^ua_i(k+i)f_{k+i-1}+\sum_{j=0}^vb_jf_{k+j}=c_{k,0}f_{k-1}+\cdots+c_{k,w}f_{k+w}+c_{k,w+1}f_{k+w+1} \ \ \ \text{for} \ \ \ k\ge 0,$$ with $c_{0,0}=0$.
We note that $c_{k,l}$ depends only on $a(z),b(z)$.
Notice that $c_{k,w+1}$ is $a_u(k+u)$ if $u-2>v-1$, $b_v$ if $u-2<v-1$ and $a_u(k+u)+b_v$ if $u-2=v-1$.
Then the assumption~\eqref{assump a,b} ensures that $\min\{k\ge 0\mid c_{k',w+1}\neq 0 \ \text{for all} \ k'\ge k\}=0$
and thus the $K$-linear map:
\begin{align}  \label{syokiti}
K^{w+1}\longrightarrow \{f\in (1/z)\cdot K[[1/z]]\mid L\cdot f\in K[z]\}; \ \ (f_0,\ldots,f_w)\mapsto \sum_{k=0}^{\infty} \dfrac{f_k}{z^{k+1}},
\end{align}
where, for $k\ge w+1$, $f_k$ is determined inductively by 
\begin{align} \label{recurrence general}
\sum_{i=0}^ua_i(k+i)f_{k+i-1}+\sum_{j=0}^vb_jf_{k+j}=0 \ \ \text{for} \ \ k\ge 0 
\end{align}
is an isomorphism.
This completes the proof of Lemma~$\ref{solutions}$.
\end{proof}

Let us fix Laurent series $f_0(z),\ldots,f_{w}(z)\in (1/z)\cdot K[[1/z]]$ such that these series are linearly independent over $K$ and $L\cdot f_j(z)\in K[z]$ for $0\le j \le w$.
We denote the formal integration associated to $f_j$ by $\varphi_{f_j}=\varphi_j$.
For a nonnegative integer $n$, we denote the $n$-th Rodriges operator associated with $L$ ({\it refer} \cite[Equation $(2.5)$]{A-B-V} and \cite[Definition $3.1$]{Kaw}) by 
\begin{align*} 
R_n=\dfrac{1}{n!}\left(\dfrac{d}{dz}+\dfrac{b(z)}{a(z)}\right)^na(z)^n.
\end{align*}
The differential operator $R_n$ can be decomposed as follows:
\begin{lemma} \label{decompose R} $(${\rm{\cite[Lemma~$4.3$]{Kaw}}}$)$
Let $w(z)$ be a non-zero solution of the differential operator $L$ in a differential extension $\mathcal{K}$ of $K(z)$. 
In the ring $\mathcal{K}[\tfrac{d}{dz}]$, we have the equalities: 
$$R_n=\dfrac{1}{n!}w(z)^{-1}\dfrac{d^n}{dz^n}w(z)a(z)^n=\dfrac{1}{n!}R_{1}(R_{1}+a^{\prime}(z))\cdots (R_{1}+(n-1)a^{\prime}(z)).$$
\end{lemma}
Making use of the Rodrigues operator associated with $L$, we construct explicit Pad\'{e} approximants of the family $(f_j)_j$.
The following theorem is a particular case of \cite[Theorem~$4.2$,\ Lemma~$5.1$]{Kaw}.
\begin{theorem} \label{Pade fj}
For a nonnegative integer $\ell$, define polynomials 
\begin{align}
P_{n,\ell}(z)=R_{n}\cdot z^{\ell}, \ \ \ Q_{n,j,\ell}(z)=\varphi_{j}\left(\dfrac{P_{n,\ell}(z)-P_{n,\ell}(t)}{z-t}\right) \ \ \text{for} \ \ 0\le j \le w. \label{main Pade}
\end{align}
Then the following properties hold.

\medskip

$(i)$ The vector of polynomials $(P_{n,\ell}(z),Q_{n,j,\ell}(z))$ forms a weight $n$ Pad\'{e}-type approximant for $f_0(z), \dots, f_{w}(z)$.

\medskip

$(ii)$ Put the Pad\'{e}-type approximation of $(f_j)_j$ by 
\begin{align}\label{remainder}
\mathfrak{R}_{n,j,\ell}(z)=P_{n,\ell}(z)f_j(z)-Q_{n,j,\ell}(z) \ \ \text{for} \ \ 0\le j \le w.
\end{align}
We have \begin{align*} 
\mathfrak{R}_{n,j,\ell}(z)=(-1)^n\sum_{k=n}^{\infty}\binom{k}{n}\dfrac{\varphi_{j}(t^{k+\ell-n}{a}(t)^{n})}{z^{k+1}}.
\end{align*}
\end{theorem}
\begin{proof}
$(i)$ The statement is derived by applying \cite[Theorem~$4.2$]{Kaw} in the special case where $d=1$ and $D_1=L$.

\medskip

$(ii)$ By definition of $\mathfrak{R}_{n,j,\ell}$, we see $\mathfrak{R}_{n,j,\ell}(z)=\varphi_j\left(P_{n,\ell}(t)/(z-t)\right)$. Substituting the expansion $1/(z-t)=\sum_{k=0}^{\infty}t^k/z^{k+1}$ into this expression, we obtain 
$$\mathfrak{R}_{n,j,\ell}(z)=\sum_{k=n}^{\infty}\dfrac{\varphi_j(t^kP_{n,\ell}(t))}{z^{k+1}}.$$
The above equality implies it is sufficient to prove 
\begin{align} \label{suff}
\varphi_j(t^kP_{n,\ell}(t))=(-1)^n\binom{k}{n}\varphi_j(t^{\ell+k-n}{a}(t)^{n}) \ \ \text{for} \ \ k\ge n.
\end{align}
Now we fix an integer $k\ge n$. 
To prove equation~\eqref{suff}, we define the differential operator $\mathcal{E}_{a,b}=\tfrac{d}{dt}+b(t)/a(t)\in K(t)[\tfrac{d}{dt}]$. Notice $L^{*}=\mathcal{E}_{a,b}a(t)$.
By \cite[Proposition~$3.2 \, (i)$]{Kaw}, there exists a set of integers $\{c_{n,k,i};\,0\le i \le n\}$ such that $c_{n,k,n}=(-1)^nk(k-1)\cdots (k-n+1)$ and $$t^k\mathcal{E}^n_{a,b}=\sum_{i=0}^nc_{n,k,i}\mathcal{E}_{a,b}^{n-i}t^{k-i}.$$
This implies 
\begin{align} \label{bunkai 1}
t^kP_{n,\ell}(t)=\dfrac{t^k}{n!}\mathcal{E}^n_{a,b}a(t)^n\cdot t^{\ell}=\sum_{i=0}^n\dfrac{c_{n,k,i}}{n!}\mathcal{E}_{a,b}^{n-i}t^{k-i}a(t)^{n}\cdot t^{\ell}.
\end{align}
From \cite[Proposition~$3.2 \, (ii)$]{Kaw}, we know
\begin{align} \label{bunkai 2}
\mathcal{E}_{a,b}^{n-i}t^{k-i}a(t)^{n}\cdot t^{\ell}\in L^{*}\cdot K[t] \ \ \text{for} \ \ 0\le i \le n-1.
\end{align}
Combining Equations \eqref{bunkai 1}, \eqref{bunkai 2} together with $L^{*}\cdot K[t]\subseteq {\rm{ker}}\,\varphi_j$ (refer Proposition~$\ref{inc}$) deduces 
$$\varphi_j(t^kP_{n,\ell}(t))=\varphi_j\left(\dfrac{c_{n,k,n}}{n!}t^{k+\ell-n}a(t)^n\right)=(-1)^n\binom{k}{n}\varphi_j(t^{k+\ell-n}a(t)^n).$$
This establishes equation~\eqref{suff} and completes the proof of $(ii)$. 
\end{proof} 

Let $n$ be a nonnegative integer. Denote the determinant of the matrix formed by the Pad\'{e} approximants of $(f_j)_j$ in Theorem~$\ref{Pade fj}$ by 
$$\Delta_n(z)={\rm{det}}
\begin{pmatrix}
P_{n,0}(z) & P_{n,1}(z) & \cdots & P_{n,w+1}(z)\\
Q_{n,0,0}(z)& Q_{n,0,1}(z) & \cdots & Q_{n,0,w+1}(z)\\
\vdots & \vdots & \ddots & \vdots\\
Q_{n,w,0}(z)& Q_{n,w,1}(z) & \cdots & Q_{n,w,w+1}(z)
\end{pmatrix}.
$$ 
In the next section, we consider the non-vanishing of $\Delta_n(z)$.

\section{Linear independence of Pad\'{e}-type approximants} \label{Pade linear independence}
In this section, we keep the notation in Subsection $\ref{Pade subsection}$. Recall the polynomials $a(z),b(z)\in K[z]$ defined in equation~\eqref{a,b} which satisfy \eqref{w} and \eqref{assump a,b}. 
Recall $w=\max\{{\rm{deg}}\,a-2,{\rm{deg}}\,b-1\}$.
This section is devoted to prove the next theorem.
\begin{theorem} \label{nonvanish Delta}
For the polynomials $a(z),b(z)$, we assume 
\begin{align} 
&na'(z)+b(z) \ \text{and} \ a(z) \ \text{are coprime for any positive integer} \ n, \label{key assump2}\\
&a_u((k+1)u+n)+b_{v}\neq 0 \ \ \text{for any} \ \ k,n\ge 0 \ \ \text{if} \ \ u-1=v. \label{important assump} 
\end{align}
Then $\Delta_n(z)\in K\setminus\{0\}$ for any $n\ge 0$.
\end{theorem}
Our strategy of the proof of Theorem~$\ref{nonvanish Delta}$ is the following. 
By definition $\Delta_n(z)$ is a polynomial.
We show in Proposition~\ref{Delta and Theta}, which is essentially an application of \cite[Lemma 4.2~(ii)]{DHK3}, that this polynomial is a constant, and we reduce the problem to showing that another determinant ${\rm{det}}\,\mathcal{M}_n$ is non-zero. 
This last property, established in Subsection~\ref{finish nonvanish}, will imply Theorem~\ref{nonvanish Delta}. 
In order to prove above results, we prepare the following lemma. 
\begin{lemma} \label{sufficient condition}
Denote the differential operator $-a(z)\tfrac{d}{dz}+b(z)$ by $L$. 
Let $k\in \N$ and $P(t)\in K[t]\setminus\{0\}$. 
For the polynomials $a(z),b(z)$, we assume equation~\eqref{important assump} holds. 
Then we have $${\rm{deg}}\, (L^{*}+ka'(t))\cdot P(t)={\rm{deg}}\,P+w+1.$$
\end{lemma}
\begin{proof}
We may assume $P(t)=t^N$ for a nonnegative integer $N$.
Put $$\mathcal{P}(t)=\left(L^*+ka'(t)\right)\cdot t^N.$$ 
Then a direct computation yields that ${\rm{deg}}\, \mathcal{P} \le N+w+1$ and the coefficient of $t^{N+w+1}$ of $\mathcal{P}$ is $a_u(u(k+1)+N)+b_v$ if $u-1=v$, $b_v$ if $v>u-1$ and $a_u((k+1)u+N)$ if $u-1>v$. 
The assumption~\eqref{important assump} ensures the assertion.
This completes the proof of the statement. 
\end{proof}
Let $n$ be a nonnegative integer. Recall the polynomials $P_{n,\ell}(z),Q_{n,j,\ell}(z)$ defined in equation~\eqref{main Pade}.
Combining Lemmas $\ref{decompose R}$ and $\ref{sufficient condition}$ yields   
\begin{align} \label{deg of Pl}
{\rm{deg}}\, P_{n,\ell}=\ell+n(w+1).
\end{align}
Put the following $w+1$ by $w+1$ matrix by 
\begin{align*} 
\mathcal{M}_n=
\begin{pmatrix}
\varphi_{0}(a(t)^n) & \cdots & \varphi_{{0}}(t^{w}a(t)^n)\\
\vdots & \ddots & \vdots\\
\varphi_{{w}}(a(t)^n) & \cdots & \varphi_{{w}}(t^{w}a(t)^n)
\end{pmatrix}\in {\rm{Mat}}_{w+1}(K).
\end{align*}
We now prove $\Delta_n(z)$ is a constant. 
\begin{proposition} \label{Delta and Theta}
There exists a constant $c\in K$ such that $c$ is non-zero under the assumption~\eqref{important assump} and 
$$\Delta_n(z)=c \cdot {\rm{det}}\, \mathcal{M}_n.$$
In particular, we have $\Delta_n(z)\in K$.
\end{proposition}
\begin{proof}
Note that this proposition is a particular case of \cite[Proposition~$5.2$]{Kaw}. 
For the matrix in the definition of $\Delta_n(z)$, adding $-f_{j}(z)$ times first row to $j+1$ th row for each $0\le j \le w$, 
                     $$ 
                     \Delta_n(z)=(-1)^{w}{\rm{det}}
                     {\begin{pmatrix}
                     P_{n,0}(z) & \dots &P_{n,w+1}(z)\\
                     \mathfrak{R}_{n,0,0}(z) & \dots & \mathfrak{R}_{n,0,w+1}(z)\\
                     \vdots    & \ddots & \vdots  \\
                     \mathfrak{R}_{n,w,0}(z) & \dots & \mathfrak{R}_{n,w,w+1}(z)
                     \end{pmatrix}}. 
                     $$ 
We denote the $(s,t)$th cofactor of the matrix in the right hand side of above equality by $\Delta_{s,t}(z)$.
Then we have, developing along the first row 
\begin{align} \label{formal power series rep delta}
\Delta_n(z)=(-1)^{w}\left(\sum_{l=0}^{{{w}}}P_{n,\ell}(z)\Delta_{1,\ell+1}(z)\right).
\end{align} 
The property of the Pad\'{e} approximation ${\rm{ord}}_{\infty}\, {{\mathfrak{R}}}_{n,j,\ell}(z)\ge n+1$ for $0\le j \le w$, $0 \le \ell \le w+1$ implies
$$
{\rm{ord}}_{\infty}\,\Delta_{1,\ell+1}(z)\ge (n+1)(w+1) \ \ \text{for} \ \ 0\le \ell \le w.
$$
Combining equation~\eqref{deg of Pl} and above inequality yields
$$P_{n,\ell}(z)\Delta_{1,\ell+1}(z)\in (1/z)\cdot K[[1/z]] \ \ \text{for} \ \ 0\le \ell \le w,$$
and 
$$P_{n,w+1}(z)\Delta_{1,{{w+1}}}(z)\in K[[1/z]].$$
Note that in above relation, the constant term of $P_{w}(z)\Delta_{1,{{w+1}}}(z)$ is 
\begin{equation} \label{what const} 
``\text{Coefficient of} \ z^{(n+1)(w+1)} \ \text{of} \ P_{n,w+1}(z)'' \cdot ``\text{Coefficient of} \ 1/z^{(n+1)(w+1)} \ \text{of} \ \Delta_{1,{{w+1}}}(z)''.
\end{equation}  
Notice that the coefficient of $z^{(n+1)(w+1)}$ of $P_{n,w+1}(z)$ is non-zero under the assumption~\eqref{important assump} by Lemma~$\ref{sufficient condition}$.
equation $(\ref{formal power series rep delta})$ implies $\Delta_n(z)$ is a polynomial in $z$ with non-positive valuation with respect to ${\rm{ord}}_{\infty}$. Thus, it has to be a constant.
Finally, by Theorem~$\ref{Pade fj}$ $(ii)$, the coefficient of $1/z^{(n+1)(w+1)}$ of  $\Delta_{1,{{w+1}}}(z)$ is 
{\small{\begin{align*}
{\rm{det}}
\begin{pmatrix}
(-1)^n\varphi_{0}(a(t)^n) & \ldots & (-1)^n\varphi_{0}(t^{w}a(t)^n) \\
\vdots & \ddots & \vdots\\
(-1)^n\varphi_{w}(a(t)^n) & \ldots & (-1)^n\varphi_{w}(t^{w}a(t)^n) \\  
\end{pmatrix}
=(-1)^{wn}\cdot {\rm{det}}\,\mathcal{M}_n.
\end{align*}}}
Combining Equations $(\ref{formal power series rep delta})$, $(\ref{what const})$ and above equality yields the assertion.
This completes the proof of Proposition~$\ref{Delta and Theta}$.
\end{proof}

\subsection{Study of kernels of $\varphi_j$}
Let $n$ be a nonnegative integer. Denote $K[t]_{\leq n} \subset K[t]$ the $K$-vector space of polynomials of degree at most $n$.
In this subsection we consider the kernel of $\varphi_j$ and prove the following crucial lemma.
\begin{lemma} \label{kernel}
We have $$\bigcap_{j=0}^{w}{\rm{ker}}\,\varphi_{j}=L^*\cdot K[t].$$
\end{lemma}
\begin{proof}
Denote the $K$-vector space $\cap_{j=0}^{w}{\rm{ker}}\,\varphi_{j}$ by $W$.
Since $L\cdot f_j(z)\in K[z]$, Proposition~$\ref{inc}$ yields $L^{*}\cdot K[t]\subseteq W$. Let us show the opposite inclusion.   
Let $P(t)\in W$. Applying Lemma~$\ref{sufficient condition}$ for $k=0$, we see that there exists a polynomial $\tilde{P}(t)\in L^{*}\cdot K[t]$ such that $P(t)-\tilde{P}(t)\in K[t]_{\le w}$.
This implies that it is sufficient to prove the following equality:
\begin{align} \label{determine ker}
W \bigcap K[t]_{\le w}=\{0\}.
\end{align} 
Put $f_j(z)=\sum_{k=0}^{\infty}f_{j,k}/z^{k+1}$. By the definition of $\varphi_j$,
$$\mathcal{M}_0=
\begin{pmatrix}
f_{0,0} &  \cdots & f_{0,w}\\
\vdots & \ddots & \vdots\\
f_{w,0} &\cdots & f_{w,w}
\end{pmatrix}\in {\rm{Mat}}_{w+1}(K).
$$
Since the Laurent series $f_j$ is determined by $(f_{j,k})_{0\le k \le w}$ (see the $K$-isomorphism in equation~\eqref{syokiti}) and the Laurent series $\{f_j(z)\}_{0\le j \le w}$ are linearly independent over $K$, we have ${\rm{det}}\,\mathcal{M}_0 \neq 0$.
For $P(t)=\sum_{k=0}^{w}p_jt^j \in W \cap K[t]_{\le w}$, we put $\boldsymbol{p}={}^t(p_0,\ldots,p_{w})$. 
By the definition of $W$, we have 
$$\mathcal{M}_0 \cdot \boldsymbol{p}=\boldsymbol{0},$$
and thus $\boldsymbol{p}=\boldsymbol{0}$. This implies $P=0$ and equation~\eqref{determine ker} holds. 
\end{proof}
\begin{remark}
The non-vanishing of the determinant of $\mathcal{M}_0$ yields the non-vanishing of $\Delta_0(z)$.
\end{remark}

\subsection{Proof of Theorem~$\ref{nonvanish Delta}$} \label{finish nonvanish}
Let $n$ be a positive integer. By Proposition~$\ref{Delta and Theta}$, it is sufficient to show the non-vanishing of ${\rm{det}}\, \mathcal{M}_n$ 
to prove Theorem~$\ref{nonvanish Delta}$.
Let $\boldsymbol{q}={}^t(q_0,\ldots,q_{w})\in K^{w}$ such that 
$$\mathcal{M}_n\cdot \boldsymbol{q}=\boldsymbol{0}.$$ 
Put $Q(t)=\sum_{k=0}^{w}q_kt^k\in K[t]$. 
Then the linearity of $\varphi_{j}$ derives $\varphi_{j}(Q(t)a(t)^n)=0$ for $0\le j \le w$ and thus, using Lemma~$\ref{kernel}$,
\begin{align} \label{belong to kernel}
Q(t)a(t)^n\in \bigcap_{j=0}^{w}{\rm{ker}}\,\varphi_{j}=L^{*}\cdot K[t].
\end{align}
We complete the proof of Theorem~$\ref{nonvanish Delta}$ by showing $Q(t)=0$.
In order to prove $Q(t)=0$, we prepare the following key lemma.
\begin{lemma} \label{nonvanish key1}
Let $a(t),b(t)\in K[t]$. Assume $Na'(t)+b(t)$ and $a(t)$ are coprime for any positive integer $N$.
Let $n$ be a positive integer and $P(t),Q(t)\in K[t]$ such that 
\[
Q(t)a(t)^n=\left(\dfrac{d}{dt}+\dfrac{b(t)}{a(t)}\right)a(t)\cdot P(t).
\]
Then $P(t)$ is divisible by $a(t)^n$. 
\end{lemma}
\begin{proof}
Let us prove the statement by induction on $n$.
Let $n=1$. Then $$Q(t)a(t)=(a'(t)+b(t))P(t)+a(t)P'(t).$$ 
Since $a'(t)+b(t)$ is coprime with $a(t)$, the polynomial $P(t)$ is divisible by $a(t)$. 
Assume the claim holds for $n\ge1$. 
Let us consider the case $n+1$. 
The induction hypothesis yields $P(t)$ is divisible by $a(t)^n$. Put $P(t)=a(t)^n\tilde{P}(t)$.
Then $$Q(t)a(t)^{n+1}=\left((n+1)a'(t)+b(t)\right)\tilde{P}(t)a(t)^n+a(t)^{n+1}\tilde{P}'(t).$$
Combining the hypothesis of $a(t),b(t)$ and this equality yields $\tilde{P}(t)$ is divisible by $a(t)$ and thus $P(t)$ is divisible by $a(t)^{n+1}$.
\end{proof}
We now finish the proof of Theorem~$\ref{nonvanish Delta}$.
\begin{proof}[\textbf{Proof of Theorem~$\ref{nonvanish Delta}$}]
Recall $a(z),b(z)$ be polynomials defined in equation~\eqref{a,b} satisfying the assumptions \eqref{w}, \eqref{key assump2} and \eqref{important assump}.
For a vector $\boldsymbol{q}={}^t(q_0,\ldots,q_{w})$ satisfying $\mathcal{M}_n\cdot \boldsymbol{q}=\boldsymbol{0}$, we put $Q(t)=\sum_{k=0}^{w}q_kt^k$. Assume $Q(t)\neq 0$. 
From equation~\eqref{belong to kernel}, there exists a non-zero polynomial $P(t)\in K[t]$ such that $$Q(t)a(t)^n=L^{*}\cdot P(t).$$
Using the assumption~\eqref{key assump2}, by Lemma~$\ref{nonvanish key1}$, the polynomial $P(t)$ must be divisible by $a(t)^n$ and thus we have ${\rm{deg}}\,P\ge nu$. 
Using the assumption~\eqref{important assump} together with Lemma~$\ref{sufficient condition}$ for $k=0$, we have $${\rm{deg}}\, L^{*}\cdot P(t)\ge nu+w+1.$$
However, by the definition of $Q(t)$, we know $${\rm{deg}}\,Q(t)a(t)^n\le nu+w.$$ 
This leads to a contradiction, as the degrees cannot simultaneously satisfy both conditions.
Thus we conclude $Q(t)=0$, completing the proof of Theorem~$\ref{nonvanish Delta}$.
\end{proof}

\begin{remark}
Let $K$ be an algebraic number field, and let $a(z), b(z) \in K[z]$ satisfy the assumptions~\eqref{a,b}, \eqref{w} and~\eqref{assump a,b}.  
Consider the differential operator $L = -a(z)\tfrac{d}{dz} + b(z)$.
In Theorem~\ref{main}, we apply Theorem~\ref{nonvanish Delta} to the case where ${\deg}\,a = m$ and ${\deg}\,b \le m-1$, under the assumptions~\eqref{first}, \eqref{second} and~\eqref{third}, and to Laurent series $f_j(z) \in (1/z)\cdot K[[1/z]]$ that are linearly independent over~$K$ and satisfy $L \cdot f_j \in K[z]$.  
In this setting, as mentioned in Remark~\ref{fj are G functions}, the Laurent series~$f_j$ are all $G$-functions.

In~\cite{An2}, Andr\'{e} introduced the notion of \emph{arithmetic Gevrey series} as a general framework encompassing the theories of $E$-functions and $G$-functions originally developed by Siegel~\cite{Siegel1} in his study of transcendental number theory (refer \cite{An3}).  
It would be interesting to explore possible applications of Theorem~\ref{nonvanish Delta} to the study of the arithmetic properties of other classes of arithmetic Gevrey series $f(z) \in (1/z)\cdot K[[1/z]]$ satisfying $L \cdot f \in K[z]$.
\end{remark}

\section{Estimates} \label{Estimate}
Recall notation in Subsection $\ref{subsection theorem}$.
We fix a positive integer $m\ge 2$. Let $K$ be an algebraic number field and $a(z),b(z)\in K[z]$ where $a(z)$ is a monic polynomial satisfying $${\rm{deg}}\,a=m \ \ \ \text{and} \ \ {\rm{deg}}\,b\le m-1.$$
Assume that $a(z)$ decomposes over $K$, and let $\alpha_1,\ldots,\alpha_m\in K$ denote its roots, counted with multiplicity. Set
$$a(z)=\sum_{i=0}^ma_iz^i, \ \ b(z)=\sum_{j=0}^{m-1}b_jz^j.$$
Assume \eqref{first}, \eqref{second}, and \eqref{third}. 
Denote the differential operator of order $1$ $$L=-a(z)\dfrac{d}{dz}+b(z).$$ 
The assumption~\eqref{third}, together with Lemma~$\ref{solutions}$, ensures that there exist Laurent series $f_0(z),\ldots,f_{m-2}(z)\in (1/z)\cdot K[[1/z]]$ such that $f_j$ are linearly independent over $K$ and 
\begin{align} \label{property fj}
L\cdot f_j(z)\in K[z] \ \ \text{for} \ \ 0\le j \le m-2.
\end{align}
We now compute the exact form of $f_j(z)$. Recall $s_i:=b(\alpha_i)/a'(\alpha_i)$ for $1\le i \le m$ and the quantity $b_{m-1}=s_1+\ldots+s_{m}$ is not in $\Z_{<-1}$.
\begin{lemma} \label{formal solutions1}
For $0\le j \le m-2$, the formal Laurent series
\[
f_j(z):=\prod_{i=1}^m\left(1-\frac{\alpha_i}{z}\right)^{s_i}\cdot 
\frac{1}{z^{j+1}}
F^{(m)}_D\!\left(
b_{m-1}+j+1,\,
1+s_1,\ldots,1+s_m,\,
b_{m-1}+j+2;\,
\frac{\alpha_1}{z},\ldots, \frac{\alpha_m}{z}
\right)\footnote{In case of $b_{m-1}=-1$, we note that $f_0(z)=\tfrac{1}{z}\prod_{i=1}^m\left(1-\tfrac{\alpha_i}{z}\right)^{s_i}$.},
\]
are linearly independent over $K$ and satisfy relation~\eqref{property fj}.
\end{lemma}
\begin{proof}
Since ${\rm{ord}}_{\infty}\,f_j=j+1$, it is easy to see that $f_j$ are linearly independent over $K$. 
Let us show $f_j$ satisfy~\eqref{property fj}. 
Set 
\begin{align*}
\Phi(z) &= \prod_{i=1}^m\left(1-\frac{\alpha_i}{z}\right)^{s_i}, \\ 
g_j(z) &= \frac{1}{z^{j+1}} 
F^{(m)}_D\!\left(
b_{m-1}+j+1,\,
1+s_1,\ldots,1+s_m,\,
b_{m-1}+j+2;\,
\frac{\alpha_1}{z},\ldots,\frac{\alpha_m}{z}
\right).
\end{align*}
It is sufficient to prove that each $f_j(z)=\Phi(z)g_j(z)$ satisfies
\begin{equation} \label{diff eq.}
\left(\frac{d}{dz}-\frac{b(z)}{a(z)}\right)\Phi(z)g_j(z)
= \frac{(b_{m-1}+j+1)z^{m-j-2}}{a(z)}.
\end{equation}
A direct computation yields the following:
\begin{align}
\frac{d}{dz}\Phi(z) 
   &= \frac{\Phi(z)}{z}\sum_{j=1}^m\frac{\alpha_js_j}{\,z-\alpha_j}, 
   \label{d Phi}\\[6pt]
\Bigl(-z\frac{d}{dz}+b_{m-1}\Bigr) g_j(z) 
   &= -\frac{b_{m-1}+j+1}{z^{j+1}}
      \prod_{i=1}^m\left(1-\frac{\alpha_i}{z}\right)^{-1-s_i} \notag\\
   &= -\frac{(b_{m-1}+j+1)z^{m-j-1}}{a(z)\Phi(z)}. \notag
\end{align}
From the second equality we obtain
\begin{equation} \label{d g_j}
g'_j(z) = \frac{b_{m-1}}{z}\,g_j(z)
          + \frac{(b_{m-1}+j+1)z^{m-j-2}}{a(z)\Phi(z)}.
\end{equation}
Now, applying \eqref{d Phi} and \eqref{d g_j}, we compute
\begin{align*}
\left(\frac{d}{dz}-\frac{b(z)}{a(z)}\right)\Phi(z)g_j(z)
&= \Phi'(z)g_j(z) + \Phi(z)g'_j(z) - \frac{b(z)}{a(z)}\Phi(z)g_j(z)\\[4pt]
&= \left[
     \frac{1}{z}\left(\sum_{j=1}^m\frac{\alpha_js_j}{z-\alpha_j}+b_{m-1}\right)
     -\frac{b(z)}{a(z)}
   \right]\Phi(z)g_j(z)
   + \frac{(b_{m-1}+j+1)z^{m-j-2}}{a(z)}.
\end{align*}
Finally, observe the identity
\[
\frac{zb(z)}{a(z)}
   = z\sum_{j=1}^m\frac{s_j}{z-\alpha_j}
   = \sum_{j=1}^m\left(\frac{\alpha_js_j}{z-\alpha_j}+s_j\right)
   = \sum_{j=1}^m\frac{\alpha_js_j}{z-\alpha_j}+b_{m-1},
\]
which implies
\[
\frac{1}{z}\left(\sum_{j=1}^m\frac{\alpha_js_j}{z-\alpha_j}+b_{m-1}\right)
-\frac{b(z)}{a(z)}=0.
\]
Therefore, \eqref{diff eq.} holds.  
This completes the proof of Lemma~\ref{formal solutions1}.
\end{proof}
Now we fix the Laurent series $f_j(z)$ in Lemma~\ref{formal solutions1} and denote $\varphi_{f_j}$ by $\varphi_j$.
We define the polynomials
\begin{align} \label{P,Q}
P_{n,\ell}(z)=\dfrac{1}{n!}\left(\dfrac{d}{dz}+\dfrac{b(z)}{a(z)}\right)^na(z)^n\cdot z^{\ell}, \ \ \ Q_{n,j,\ell}(z)=\varphi_{j}\left(\dfrac{P_{n,\ell}(z)-P_{n,\ell}(t)}{z-t}\right) \ \ \text{for} \ \ 0\le j \le m-2.
\end{align}
By Theorem~$\ref{Pade fj} \, (i)$, the vector of polynomials $(P_{n,\ell}(z),Q_{n,j,\ell}(z))_{0\le j \le m-2}$ is a weight $n$ Pad\'{e}-type approximant of $(f_j)_j$.
Denote the Pad\'{e}-type approximations of $(f_j)_j$ by 
\begin{align} \label{main remainder}
\mathfrak{R}_{n,j,\ell}(z)=P_{n,\ell}(z)f_j(z)-Q_{n,j,\ell}(z) \ \ \text{for} \ 0\le j \le m-2.
\end{align}
A direct computation shows the algebraic function 
$$w(z)=\prod_{i=1}^m(z-\alpha_i)^{s_i}$$
is a non-zero solution of the differential operator $L$ $($\rm{refer \cite[Proposition~$3$ (i)]{F-R}}$)$.
Notice that Lemma~$\ref{decompose R}$ implies 
\begin{align} \label{express P}
P_{n,\ell}(z)=\dfrac{w(z)^{-1}}{n!} \dfrac{d^n}{dz^n}w(z)\cdot a(z)^nz^{\ell}
\end{align}
and, applying the Leibniz formula to equation~\eqref{express P} along with the expression $a(z)=\prod_{i=1}^m(z-\alpha_i)$ and the definition of $w(z)$, the following identity holds.
\begin{lemma} \label{explicit P}
We have
\begin{align*}
P_{n,\ell}(z)=\sum_{k=0}^n(-1)^k\sum_{\substack{0\le k_1,\ldots,k_m\le k \\ \sum{k_i}=k}}\sum_{\substack{0\le j_1,\ldots,j_{m+1}\le n-k \\ \sum{j_i}=n-k}}
\binom{\ell}{j_{m+1}}\left[\prod_{i=1}^m\dfrac{(-s_i)_{k_i}}{k_i!}\binom{n}{j_i}(z-\alpha_i)^{n-j_i-k_i}\right]z^{\ell-j_{m+1}},
\end{align*}
with the convention $\binom{\ell}{j_{m+1}}=0$ if $j_{m+1}>\ell$. 
\end{lemma}
\begin{proof}
The Leibniz formula yields
\begin{align} \label{Leibniz}
\dfrac{d^n}{dz^n}\cdot w(z)a(z)^nz^{\ell}&=\sum_{k=0}^n\binom{n}{k}\dfrac{d^k}{dz^{k}}\cdot w(z)\times \dfrac{d^{n-k}}{dz^{n-k}}\cdot a(z)^nz^{\ell}.
\end{align}
From the definition of $w(z)$ and the equality $a(z)=\prod_{i=1}^m(z-\alpha_i)$, we compute
\begin{align*}
&\dfrac{d^k}{dz^{k}}\cdot w(z)=w(z)\sum_{\substack{0\le k_1,\ldots,k_m\le k \\ \sum k_i=k}}(-1)^k\dfrac{k!}{k_1!\cdots k_m!}\prod_{i=1}^m \dfrac{(-s_i)_{k_i}}{(z-\alpha_i)^{k_i}},\\
&\dfrac{d^{n-k}}{dz^{n-k}}\cdot a(z)^nz^{\ell}=(n-k)!\sum_{\substack{0\le j_1,\ldots,j_{m+1}\le n-k \\ \sum{j_i}=n-k}}\binom{\ell}{j_{m+1}}z^{\ell-j_{m+1}}\prod_{i=1}^m\binom{n}{j_i}(z-\alpha_i)^{n-j_i}.
\end{align*}
Substituting \eqref{Leibniz} into equation~\eqref{express P} together with the above equalities, we obtain the expansion of $P_{n,\ell}(z)$.
\end{proof}

\subsection{Absolute values of the Pad\'{e} approximants}
Let $v$ be a place of $K$. 
In this subsection we describe the asymptotic behavior, as $n$ tends to infinity, of the $v$-adic absolute values of the polynomials $P_{n,\ell}(z)$ and $Q_{n,j,\ell}(z)$ evaluated at a fixed $\beta\in K$.

\medskip

Now we use the following notations. 
Denote the $v$-adic Weil height of the $m$-tuple of algebraic number $\boldsymbol{\alpha}=(\alpha_1,\ldots,\alpha_m)$ by 
${\rm{H}}_v(\boldsymbol{\alpha})=\exp({\rm{h}}_v(\boldsymbol{\alpha}))$.
For a polynomial $P=\sum_{k=0}^np_kz^k \in K[z]$ and a valuation $v$ of $K$, we denote the the maximum $v$-adic absolute value of its coefficients of $P$ by 
\[
||P||_v:=\max_{0\le k \le n}\{|p_k|_v\}.
\] 
The floor function is denoted by $\lfloor \cdot \rfloor$. 
For a nonnegative integer $n$, $s\in \Q$ and $b\in \Q\setminus\Z_{\le-1}$, we denote
\begin{align*}
\mu_n(s)&={\rm{den}}(s)^n\prod_{\substack{q:\text{prime} \\ q\mid {\rm{den}}(s)}} q^{\lfloor n/q-1\rfloor},\\
d_n(b)&={\rm{den}}\left(\dfrac{1}{b+1},\ldots,\dfrac{1}{b+n+1}\right).%\footnote{Note that we have $d_n(0)$ is the least common multiple of $1,\ldots,n+1$.}
\end{align*}

The aim of this section is to prove the following proposition. 
\begin{proposition} \label{estimate PQ}
Let $v$ be a place of $K$ and $\beta\in K$. 

\medskip 

$(i)$ Assume that $v$ is non-Archimedean. Then
\begin{align*}
\log\max\{|P_{n,\ell}(\beta)|_v,|Q_{n,j,\ell}(\beta)|_v\}&\le n\left(\sum_{i=1}^m{\rm{h}}_v(\alpha_i)+(m-1)({\rm{h}}_v(\boldsymbol{\alpha})+{\rm{h}}_v(\beta))\right)\\
                                                                          &+m\sum_{i=1}^m\log |\mu_n(s_i)|_v^{-1}+\log|d_{(m-1)(n+1)}(b_{m-1})|_v+o(n),
\end{align*}
where $o(n)=0$ for all but finitely many non-Archimedean places $v$.

\medskip

$(ii)$ Assume that $v$ is Archimedean. Then
\begin{align*}
\log\max\{\left|P_{n,\ell}(\beta)\right|_v, \left|Q_{n,j,\ell}(\beta)\right|_v\} \le  n\left(\sum_{i=1}^m {\rm{h}}_v(\alpha_i)+(m-1)({\rm{h}}_v(\boldsymbol{\alpha})+{\rm{h}}_v(\beta))+m\log |4|_v+o(1)\right).
\end{align*}
\end{proposition}

\subsubsection{Proof of Proposition~\ref{estimate PQ} $(i)$}
We will use the following classical lemma to control the denominator of $(\alpha)_n/n!$ and the growth of $d_n(b)$.
\begin{lemma}\label{well-known}
Let $\alpha\in \Q$ and $b\in \Q\setminus\Z_{\le -1}$.

$(i)$ Let $n$ be a nonnegative integer. For $k=0,\dots,n$, we have
    \begin{align*}
    \mu_n(\alpha) \dfrac{(\alpha)_k}{k!}\in \Z.
    \end{align*}

\medskip

$(ii)$ Let $k$ be an integer and $n_1,n_2$ be positive integers. Then
\begin{align*}
&\mu_n(\alpha+k)=\mu_n(\alpha) \ \ \text{and} \ \ \mu_{n_1+n_2}(\alpha) \ \text{is divisible by} \ \mu_{n_1}(\alpha)\mu_{n_2}(\alpha).
\end{align*} 

\medskip

$(iii)$ We have 
\[
\limsup_{n\to \infty}\dfrac{1}{n}\log d_n(b)=\dfrac{{\rm{den}}(b)}{\varphi({\rm{den}}(b))}\sum_{\substack{1\le j \le {\rm{den}}(b_{m-1}) \\ (j,{\rm{den}}(b_{m-1}))=1}}\dfrac{1}{j},
\]
where $\varphi$ is Euler's totient function.

\medskip

$(iv)$ Let $p$ be a prime. We have 
\[
\displaystyle{\lim_{n\to \infty}} \dfrac{|d_n(b)|^{-1}_p}{n}=0.
\]
\end{lemma}
\begin{proof}
$(i)$ This property was proved in \cite[Lemma~2.2]{B}.

\medskip

$(ii)$ These properties follow directly from the definition of $\mu_n(\alpha)$.

\medskip

$(iii)$, $(iv)$ These statements were established in \cite{BKS}.
\end{proof}

\begin{lemma} \label{toy 1}
Let $v$ be a non-Archimedean place of $K$.
Let $f(z)=\sum_{k=0}^{\infty}f_k/z^{k+1}\in K[[1/z]]$ and $P(z)\in K[z]$ with ${\rm{deg}}\,P=N$.
Put $Q(z)=\varphi_f\left(\tfrac{P(z)-P(t)}{z-t}\right)$.
Then we have $$||Q||_v\le ||P||_v\cdot \max_{0\le k \le N-1}\{|f_k|_v\}.$$
\end{lemma}
\begin{proof}
Set $P(z)=\sum_{i=0}^Np_iz^i$. Then, using the identity $z^i-t^i=(z-t)\sum_{k=0}^{i-1}t^{i-1-k}z^k$ for $1\le i$, we get
\begin{align} \label{expansion Q}
Q(z)&=\varphi_f\left(\sum_{k=0}^{N-1}\left(\sum_{i=k+1}^Np_kt^{i-1-k}\right)z^k\right)=\sum_{k=0}^{N-1}\left(\sum_{i=k+1}^Np_kf_{i-1-k}\right)z^k.
\end{align}
Combining above equality with the strong triangle inequality leads us to get 
\[
||Q||_v=\max_{0\le k \le N-1}\Biggl\{\left|\sum_{i=k+1}^Np_kf_{i-1-k}\right|_v\Biggr\}\le ||P||_v\cdot \max_{0\le k \le N-1}\{|f_{k}|_v\}.
\]
\end{proof}

\begin{lemma} \label{toy 2}
Let $v$ be a non-Archimedean place of $K$. Let $f(z)=\sum_{k=0}^{\infty}f_k/z^{k+1}\in K[[1/z]]$ be a Laurent series satisfying $L\cdot f(z)\in K[z]$.
For any positive integer $n$, there exists a positive integer $D$ which depends only on $f$ such that 
\begin{align} \label{denomi f}
\max_{0\le k \le n}\{|f_k|_v\}\le  |D|^{-1}_v{\rm{H}}_v(\boldsymbol{\alpha})^n\cdot \prod_{i=1}^m|\mu_n(s_i)|^{-1}_v\cdot |d_n(b_{m-1})|^{-1}_v.
\end{align}
\end{lemma}
\begin{proof}
Put 
\[
f_j(z)=\sum_{k=0}^{\infty}\frac{f_{j,k}}{z^{k+1}}.
\]
Since the Laurent series $f(z)$ can be expressed as a $K$-linear combination of $f_j(z)$ for $0\le j \le m-2$ (cf. isomorphism~\eqref{syokiti}), it suffices to prove that
\begin{align} \label{denomi fj}
\max_{0\le k \le n}\{|f_{j,k}|_v\}\le  |b_{m-1}|^{-1}_v{\rm{H}}_v(\boldsymbol{\alpha})^n\cdot \prod_{i=1}^m|\mu_n(s_i)|^{-1}_v\cdot |d_n(b_{m-1})|^{-1}_v.
\end{align}
We have the following expansions:
\begin{align*}
&\prod_{i=1}^m\left(1-\frac{\alpha_i}{z}\right)^{s_i}
=\sum_{k=0}^{\infty}\left[\sum_{\substack{0\le k_i \\ \sum k_i=k}}
\prod_{i=1}^m \frac{(-s_i)_{k_i}}{k_i!}\alpha_i^{k_i}\right]\frac{1}{z^k},\\
&F^{(m)}_D\!\left(
b_{m-1}+j+1,\,
1+s_1,\ldots,1+s_m;\,
b_{m-1}+j+2;\,
\frac{\alpha_1}{z},\ldots,\frac{\alpha_m}{z}
\right)\\
&=\sum_{k=0}^{\infty}\left[\sum_{\substack{0\le k_i \\ \sum k_i=k}}
\prod_{i=1}^m \frac{(1+s_i)_{k_i}}{k_i!}\alpha_i^{k_i}\right]
\frac{b_{m-1}+j+1}{b_{m-1}+j+k+1}\frac{1}{z^k}.
\end{align*}

From these we conclude that, for $0\le j \le m-2$ and $k\ge 0$,
\begin{align} \label{fjk}
f_{j,k+j}
=\sum_{w=0}^k
\left[\sum_{\substack{0\le k_i \le w \\ \sum k_i=w}}
\prod_{i=1}^m \frac{(-s_i)_{k_i}}{k_i!}\alpha_i^{k_i}\right]
\cdot 
\left[\sum_{\substack{0\le l_i \le k-w \\ \sum l_i=k-w}}
\prod_{i=1}^m \frac{(1+s_i)_{l_i}}{l_i!}\alpha_i^{l_i}\right]
\cdot 
\frac{b_{m-1}+j+1}{b_{m-1}+j+k-w+1}.
\end{align}

Now let $0\le w \le k \le n$.  
For any choice of positive integers $k_i,l_i$ with $\sum k_i=w$ and $\sum l_i=k-w$, Lemma~\ref{well-known}\, $(i)$ and $(ii)$ implies
\[
\left|\prod_{i=1}^m  \frac{(-s_i)_{k_i}}{k_i!}\frac{(1+s_i)_{l_i}}{l_i!}\alpha_i^{k_i+l_i}\right|_v\le \prod_{i=1}^m|\mu_n(s_i)|^{-1}_v\, {\rm{H}}_v(\boldsymbol{\alpha})^n.
\]
Therefore, by \eqref{fjk} and the strong triangle inequality, we get
\begin{align*}
\max_{0\le k \le n} \{|f_{j,k}|_v\}&\le |b_{m-1}|_v\prod_{i=1}^m|\mu_n(s_i)|_v\, {\rm{H}}_v(\boldsymbol{\alpha})^n\cdot \max_{0\le k \le n} \Biggl\{\left|\frac{1}{b_{m-1}+k+1}\right|_v\Biggr\}\\
                                           &=  |b_{m-1}|_v\prod_{i=1}^m|\mu_n(s_i)|^{-1}_v\, {\rm{H}}_v(\boldsymbol{\alpha})^n\cdot |d_n(b_{m-1})|^{-1}_v.
\end{align*}
This completes the proof of the assertion.
\end{proof}

\begin{lemma} \label{toy 3}
Let  $v$ be a non-Archimedean place and $n$ be a positive integer. Then
\[
\log \max_{0\le \ell \le m}\{||P_{n,\ell}||_v\}\le \sum_{i=1}^m\left(\log |\mu_n(s_i)|_v^{-1}+n{\rm{h}}_v(\alpha_i)\right).
\]
\end{lemma}
\begin{proof}
Put the equality
$$(z-\alpha_i)^{n-j_i-k_i}=\sum_{w_i=0}^{n-j_i-k_i}\binom{n-j_i-k_i}{w_i}(-\alpha_i)^{n-j_i-k_i-w_i}z^{w_i} \ \ \text{for} \ 0\le k_i\le k\le n, \ 0\le j_i \le n-k_i,$$ 
into the expression of $P_{n,\ell}$ obtained in Lemma~$\ref{explicit P}$. We then have
{\small{
\begin{align} \label{Pnl expand}
P_{n,\ell}(z)=\sum_{k}(-1)^k \sum_{k_i} \sum_{j_i}\binom{\ell}{j_{m+1}}\left[\prod_{i=1}^m \dfrac{(-s_i)_{k_i}}{k_i!}\binom{n}{j_i}\sum_{w_i}\binom{n-j_i-k_i}{w_i}(-\alpha_i)^{n-j_i-k_i-w_i}z^{w_i}\right]z^{\ell-j_{m+1}}.
\end{align}
}}
Making use of Lemma~\ref{well-known} $(i)$ together with the strong triangle inequality for the above identity, we obtain the desire estimate.
\end{proof}

\begin{proof}[\textbf{Proof of Proposition~\ref{estimate PQ} $(i)$}]
We apply Lemma~\ref{toy 1} for $P=P_{n,\ell}$. Since ${\rm{deg}}\,P_{n,\ell}=\ell+(m-1)n$, using Lemma~\ref{toy 2}, the value $\log\max_{j,\ell}\{||Q_{n,j,\ell}||_v\}$ is bounded by
\[
||P_{n,\ell}||_v+(m-1)n{\rm{h}}_v(\boldsymbol{\alpha})+\sum_{i=1}^m\log|\mu_{(m-1)(n+1)}(s_i)|^{-1}_v+\log|d_{(m-1)(n+1)}(b_{m-1})|_v+|b_{m-1}|_v.
\]
Finally, Lemma~\ref{toy 3} together with the strong triangle inequality yields
\begin{align*}
&\log\max\{|P_{n,\ell}(\beta)|_v,|Q_{n,j,\ell}(\beta)|_v\}\le\\ 
&\sum_{i=1}^m\left(m\log |\mu_n(s_i)|_v^{-1}+n{\rm{h}}_v(\alpha_i)\right)+(m-1)n({\rm{h}}_v(\boldsymbol{\alpha})+{\rm{h}}_v(\beta))+\log|d_{(m-1)(n+1)}(b_{m-1})|_v+o(n),
\end{align*}
where $o(n)=0$ if $v$ does not divide ${\rm{den}}(b_m-1,s_1,\ldots,s_m)$.
This completes the proof of Proposition~\ref{estimate PQ} $(i)$.
\end{proof}

\subsubsection{Poincar\'{e}-Perron type recurrence}
We now turn attention to the proof of Proposition~\ref{estimate PQ} $(ii)$. To this end, in this subsection, let us consider the following Poincar\'{e}-type recurrence of some order $m> 0$.
\begin{align}
    \label{eq: poincare generic}
    a_m(n) u(n+m)+ a_{m-1}(n) u(n+m-1) + \cdots + a_0(n) u(n) = 0
\end{align}
for large enough $n$, where the coefficients $a_i(t)\in\C[t]$ are polynomials and $a_m(t)\neq 0$.
Then, we can apply Perron's Second Theorem below (see \cite{Perron} and \cite[Theorem~C]{Pituk}) to estimate precisely the growth of a solution of the above recurrence. 
\begin{theorem}[Perron's Second Theorem]
    \label{thm: perron second}
    Let $m$ be a positive integer. Assume that for $i=0,\dots,m$ there exist a function $a_i:\Z_{\geq 0}\rightarrow \C$ and $a_i\in\C$ such that
    \begin{align*}
        \lim_{n\to\infty} a_i(n) = a_i \in\C,
    \end{align*}
    with $a_m\neq 0$. Denote by $\lambda_1,\dots,\lambda_m$ the (not necessarily distinct) roots of the characteristic polynomial
    \begin{align*}
        \chi(z) = a_mz^m + a_{m-1}z^{m-1} + \cdots + a_0.
    \end{align*}
    Then, there exist $m$ linearly independent solutions $u_1,\dots,u_m$ of \eqref{eq: poincare generic}, such that, for each $i=1,\dots,m$,
    \begin{align*}
        \limsup_{n\to\infty} |u_i(n)|^{1/n} = |\lambda_i|.
    \end{align*}
    In particular, any solution $u$ of \eqref{eq: poincare generic} satisfies $\limsup_{n\to\infty} |u(n)|^{1/n} \leq \max_{1\leq i\leq m} \{|\lambda_i|\}$.
\end{theorem}

\begin{remark}
    In the above theorem, there are no restriction on the roots of $\chi(z)$, whereas in Poincar\'{e}'s Theorem and Perron's First Theorem, we ask that
    \begin{align}
        \label{eq: condition roots}
        |\lambda_i|\neq |\lambda_j| \textrm{ for } i\neq j,
    \end{align}
    see \cite[Theorem~A and~B]{Pituk}. 
\end{remark}
\begin{corollary} \label{estimate fk}
Let us fix an embedding $K$ into $\C$.
Let $f(z)=\sum_{k=0}^{\infty}f_k/z^{k+1}\in (1/z)\cdot \C[[1/z]]$ be a Laurent series satisfying $L \cdot f(z)\in \C[z]$. Then we have 
$$\limsup_{n\to \infty}|f_n|^{1/n}\le \max_{1\le i \le m}\{|\alpha_i|\}.$$  
\end{corollary}
\begin{proof}
Equation~\eqref{recurrence general} in Lemma~$\ref{solutions}$ yields that the vector $(f_n)_{n\ge -1}$ with $f_{-1}=0$ is a solution of the recurrence equation:
\[
\sum_{i=0}^ma_i(k+i)f_{k+i-1}+\sum_{j=0}^{m-1}b_jf_{k+j}=0 \ \ \text{for} \ \ k\ge 0.
\]
A straightforward computation yields that this recurrence equation is of Poincar\'{e} type and the characteristic polynomial $\chi(z)$ is $a(z)$. Thus Perron's Second Theorem ensures 
$$\limsup_{n\to \infty}|f_n|^{1/n}\le \max_{1\le i \le m}\{|\alpha_i|\}.$$ 
This completes the proof of Corollary~$\ref{estimate fk}$. 
\end{proof}  

\subsubsection{Proof of Proposition~$\ref{estimate PQ}$ $(ii)$}
\begin{lemma} \label{||P||}
Let $v$ be an Archimedean valuation of $K$. Define a constant $c_v(\boldsymbol{\alpha})$ which depends on $\boldsymbol{\alpha}=(\alpha_1,\ldots,\alpha_m)$ and $v$ by 
\begin{align} \label{C}
c_v(\boldsymbol{\alpha})=\sum_{i=1}^m {\rm{h}}_v(\alpha_i)+m\log |4|_v.
\end{align}
Then we have
$$\log \max_{0\le \ell \le m-1}\{||P_{n,\ell}(z)||_v\}\le nc_v(\boldsymbol{\alpha})+o(n).$$ 
\end{lemma}
\begin{proof}
We apply the estimate
\begin{align*}
&\dfrac{(s)_k}{k!}\le e^{o(n)}, \ \  \binom{n}{k}\le 2^n \ \ \text{for} \ \ s\in \Q \ \ \text{and} \ \ 0\le k \le n,
\end{align*}
together with the triangular inequality for equation~\eqref{Pnl expand}, we conclude the desire estimate.
\end{proof}

\begin{lemma} \label{trivial estimate}
Let $n$ be a positive integer and $P(z)\in K[z]$ with ${\rm{deg}}\,P=n$. Let $v$ be an Archimedean valuation of $K$ and $\beta\in K$.
Let $f(z)=\sum_{k=0}^{\infty}f_k/z^{k+1}\in (1/z)\cdot K[[1/z]]$ be a Laurent series satisfying $L \cdot f(z)\in K[z]$.
Put $$Q(z)=\varphi_f\left(\dfrac{P(z)-P(t)}{z-t}\right).$$ Then the following inequalities hold.

\medskip

$(i)$ $|P(\beta)|_v\le (n+1)||P||_v {\rm{H}}_v(\beta)^n$.

\medskip

$(ii)$ $||Q||_v\le n||P||_v{\rm{H}}_v(\boldsymbol{\alpha})^{n+o(n)}$.

\medskip

{\rm{(iii)}} $|Q(\beta)|_v\le n^2||P||_v {\rm{H}}_v(\boldsymbol{\alpha})^{n+o(n)}{\rm{H}}_v(\beta)^{n-1}$.
\end{lemma}
\begin{proof}
$(i)$ This is an easy consequence of the triangle inequality.

\medskip

$(ii)$  Equation~\eqref{expansion Q} together with the triangle inequality and the estimate
$$\max_{0\le k \le n-1} \{|f_k|_v\}\le {\rm{H}}_v(\boldsymbol{\alpha})^{n+o(n)}$$
deduced from Corollary~$\ref{estimate fk}$ allows us to obtain the desire inequality. 

\medskip

{\rm{(iii)}} Combining $(i)$ and $(ii)$ yields 
$$|Q(\beta)|_v\le n||Q||_v {\rm{H}}_v(\beta)^{n-1}\le  n^2||P||_v {\rm{H}}_v(\boldsymbol{\alpha})^{n+o(n)}{\rm{H}}_v(\beta)^{n-1}.$$
This completes the proof of Lemma~$\ref{trivial estimate}$.
\end{proof}

\begin{proof}[\textbf{Proof of Proposition~$\ref{estimate PQ}$} $(ii)$]
Let $c_v(\boldsymbol{\alpha})$ be the constant defined in \eqref{C}.
We apply Lemma~$\ref{trivial estimate}$ $(i)$ with $P=P_{n,\ell}$. 
Using Lemma~$\ref{||P||}$ and the equality ${\rm{deg}}\,P_{n,\ell}=\ell+(m-1)n$ yields
\begin{align} \label{P}
\log\,|P_{n,\ell}(\beta)|_v\le n\left(c_v(\boldsymbol{\alpha})+(m-1){\rm{h}}_v(\beta)+o(1)\right).
\end{align}
For the estimating of $Q_{n,j,\ell}$, we invoke Lemma~$\ref{trivial estimate}$ ({\rm{iii}}), which gives 
\begin{align} \label{Q} 
\log\,|Q_{n,j,\ell}(\beta)|_v \leq  n\left(c_v(\boldsymbol{\alpha})+(m-1)\left({\rm{h}}_v(\boldsymbol{\alpha})+{\rm{h}}_v(\beta)\right)+o(1)\right). 
\end{align} 
Combining inequalities \eqref{P} and \eqref{Q}, we arrive at the desired estimate.
\end{proof}

\subsection{Absolute values of the Pad\'{e} approximations}
For a rational number $s$ and a place $v$ of $K$, we define
$$
\mu_v(s)=\left\{\begin{array}{ll} 1 & \mbox{if $v\in \mathfrak{M}^{\infty}_K$ or $v\in \mathfrak{M}^{f}_K$ \& $\vert s \vert_v\leq 1$},\\
\vert {\rm{den}}(s) \vert_v\vert p\vert_v^{\tfrac{1}{p-1}} & \mbox{if $v\in \mathfrak{M}^{f}_K$ \& $\vert s \vert_v>1$ where $p$ is the prime below $v$}.\end{array}\right.
$$
\begin{proposition} \label{Remainder}
Let $v$ be a place of $K$ and $\beta\in K$. 

\medskip

$(i)$ Assume that $v$ is non-Archimedean and 
\begin{align} \label{beta}
|\beta|_v>\prod_{i=1}^m\mu_v(s_i)^{-1}\cdot {\rm{H}}_v(\boldsymbol{\alpha}).
\end{align}
Then the series $\mathfrak{R}_{n,j,\ell}(z)$ converges to an element of $K_v$ at $z=\beta$ and 
\[
\log\max_{\substack{0\le j \le m-2 \\ 0\le \ell \le m-1}} \{|\mathfrak{R}_{n,j,\ell}(\beta)|_v\} \le  n\left(-\log|\beta|_v+\sum_{i=1}^m{\rm{h}}_v(\alpha_i)+m{\rm{h}}_v(\boldsymbol{\alpha})-m\sum_{i=1}^m\log \mu_v(s_i)+o(1)\right).
\]

\medskip

$(ii)$ Assume that $v$ is Archimedean and $|\beta|_v> \max\{|\alpha_i|_v\}$. Then the series $\mathfrak{R}_{n,j,\ell}(z)$ converges to an element of $K_v$ at $z=\beta$ and   
\[
\log\max_{\substack{0\le j \le m-2 \\ 0\le \ell \le m-1}} \{|\mathfrak{R}_{n,j,\ell}(\beta)|_v\} \le n\left(-\log|\beta|_v+\sum_{i=1}^m{\rm{h}}_v(\alpha_i)+m{\rm{h}}_v(\boldsymbol{\alpha})+m\log|2|_v+o(1)\right).
\]
\end{proposition}
\begin{proof} 
Before starting the proof, we give an expansion of $\mathfrak{R}_{n,j,\ell}(z)$.
Combining the expansion 
$$t^{k+\ell}a(t)^n=\sum_{w=0}^{mn}\sum_{\substack{ k_i\le n \\ \sum{k_i=w}}}\prod_{i=1}^m\binom{n}{k_i}(-\alpha_i)^{n-k_i}t^{w+k+\ell},$$
and Theorem~$\ref{Pade fj}$ $(ii)$ implies 
\begin{align} \label{decomp R_j} 
\mathfrak{R}_{n,j,\ell}(z)=\dfrac{(-1)^n}{z^{n+1}}\sum_{k=0}^{\infty}\binom{k+n}{n}\left(\sum_{w=0}^{mn}\sum_{\substack{ k_i\le n \\ \sum{k_i}=w}}\prod_{i=1}^m\binom{n}{k_i}(-\alpha_i)^{n-k_i}f_{w+k+\ell}\right)\cdot z^{-k}.
\end{align}

\medskip

$(i)$
Let $v$ be a non-Archimedean valuation. 
By the strong triangle inequality together with Lemma~\ref{toy 2}, we obtain
\begin{align}
&\left|
\binom{k+n}{n}
\sum_{w=0}^{mn}
\sum_{\substack{k_i \le n \\ \sum k_i = w}}
\prod_{i=1}^m \binom{n}{k_i}(-\alpha_i)^{n-k_i}
f_{w+k+\ell}\,\beta^{-k}
\right|_v
\le
\label{bound coeff}\\
&\quad
\prod_{i=1}^m {\rm H}_v(\alpha_i)^{n}
\cdot {\rm H}_v(\boldsymbol{\alpha})^{mn+k+m}
\cdot \prod_{i=1}^m |\mu_{mn+k+m}(s_i)|_v^{-1}
\cdot |d_{mn+k+m}(b_{m-1})|_v^{-1}
\cdot |b_{m-1}|_v^{-1}
|\beta|_v^{-k},
\nonumber
\end{align}
for every $k \ge 0$.

By Lemma~\ref{well-known}~(ii) and the divisibility relation 
\[
d_{mn+k+m}(b_{m-1}) \mid d_{m(n+1)}(b_{m-1})\,d_{k}(b_{m-1}+m(n+1)),
\]
we have
\begin{align*}
&\prod_{i=1}^m |\mu_{mn+k+m}(s_i)|_v^{-1}\,
  |d_{mn+k+m}(b_{m-1})|_v^{-1}\\
&\qquad \le
\prod_{i=1}^m |\mu_{m(n+1)}(s_i)|_v^{-1}\,
  |d_{(m+1)n}(b_{m-1})|_v^{-1}
  \prod_{i=1}^m |\mu_{k}(s_i)|_v^{-1}\,
  |d_{k}(b_{m-1}+m(n+1))|_v^{-1}.
\end{align*}
Combining this with~\eqref{bound coeff}, we obtain
{\small
\begin{align*}
\left|
\binom{k+n}{n}
\sum_{w=0}^{mn}
\sum_{\substack{k_i\le n \\ \sum k_i = w}}
\prod_{i=1}^m \binom{n}{k_i}(-\alpha_i)^{n-k_i}
f_{w+k+\ell}\,\beta^{-k}
\right|_v
&\le
\prod_{i=1}^m |\mu_{m(n+1)}(s_i)|_v^{-1}
|d_{(m+1)n}(b_{m-1})|_v^{-1}
\prod_{i=1}^m {\rm H}_v(\alpha_i)^{n}
{\rm H}_v(\boldsymbol{\alpha})^{m(n+1)}\\
&\quad
\cdot |b_{m-1}|_v^{-1}|d_{k}(b_{m-1}+m(n+1))|_v^{-1}
{\rm{H}}_v(\boldsymbol{\alpha})^k\prod_{i=1}^m |\mu_{k}(s_i)|_v^{-1}
|\beta|_v^{-k}.
\end{align*}
}
Since Lemma~\ref{well-known}~(iv) implies that
\[
|d_{k}(b_{m-1}+m(n+1))|_v^{-1} = o(k)
\qquad (k \to \infty),
\]
and by assumption~\eqref{beta}, we deduce that
\[
\limsup_{k\to\infty}
|b_{m-1}|_v^{-1}|d_{k}(b_{m-1}+m(n+1))|_v^{-1}
{\rm{H}}_v(\boldsymbol{\alpha})^k
\prod_{i=1}^m |\mu_{k}(s_i)|_v^{-1}
|\beta|_v^{-k}
= 0.
\]
Hence $\mathfrak{R}_{n,j,\ell}(z)$ converges to an element of $K_v$ at $z=\beta$, and
\begin{align*}
\log |\mathfrak{R}_{n,j,\ell}(\beta)|_v
&\le
\log |\beta|_v^{-n-1}
+ \log \max_{k\ge 0}
\left\{
\left|
\binom{k+n}{n}
\sum_{w=0}^{mn}
\sum_{\substack{k_i\le n \\ \sum k_i = w}}
\prod_{i=1}^m \binom{n}{k_i}(-\alpha_i)^{n-k_i}
f_{w+k+\ell}\,\beta^{-k}
\right|_v
\right\}\\[2mm]
&\le
n\left(
-\log|\beta|_v
+ \sum_{i=1}^m {\rm h}_v(\alpha_i)
+ m\,{\rm h}_v(\boldsymbol{\alpha})
- m\sum_{i=1}^m \log \mu_v(s_i)
+ o(1)
\right).
\end{align*}

\medskip

$(ii)$ Let $v$ be an Archimedean valuation. 
The triangle inequality and Corollary~$\ref{estimate fk}$ lead us to 
$$\left|\sum_{w=0}^{mn}\sum_{\substack{ k_i\le n \\ \sum{k_i}=w}}\prod_{i=1}^m\binom{n}{k_i}(-\alpha_i)^{n-k_i}f_{w+k+\ell}\right|_v\le e^{o(n)}|2|^{mn}_v\prod_{i=1}^m{\rm{H}}_v(\alpha_i)^n \max\{|\alpha_i|_v\}^{mn+k}$$
for any $k\ge0$. Therefore combining \eqref{decomp R_j} and the assumption $|\beta|_v> \max\{|\alpha_i|_v\}$ yields that $\mathfrak{R}_{n,j,\ell}(z)$ converges at $\beta$ in $K_v$ and 
\begin{align*}
\left|\mathfrak{R}_{n,j,\ell}(\beta)\right|_v&\le  e^{o(n)}|\beta|^{-n}_v|2|^{mn}_v\prod_{i=1}^m{\rm{H}}_v(\alpha_i)^n \max\{|\alpha_i|_v\}^{mn}\sum_{k=0}^{\infty}\binom{k+n}{n}\left(\dfrac{\max\{|\alpha_i|_v\}}{|\beta|_v}\right)^k\\
&\le e^{o(n)}|\beta|^{-n}_v|2|^{mn}_v\prod_{i=1}^m{\rm{H}}_v(\alpha_i)^n {\rm{H}}_v(\boldsymbol{\alpha})^{mn}.
\end{align*}
Taking the logarithm, we obtain the conclusion of $(ii)$.
\end{proof}
\section{Proof of Theorem~$\ref{main}$} \label{proof}
We keep the notation in Section $\ref{Estimate}$.
For a positive integer $n$, we recall that the polynomials $P_{n,\ell}(z)$ and $Q_{n,j,\ell}(z)$ are defined in equation~\eqref{P,Q}. 
Let us fix a place $v_0$ of $K$ and let $\beta\in K$. 
Define the following $m$ by $m$ matrix $M_n$ as 
$$M_n=\begin{pmatrix}
P_{n,0}(\beta) & P_{n,1}(\beta) & \cdots & P_{n,m-1}(\beta)\\
Q_{n,0,0}(\beta)& Q_{n,0,1}(\beta) & \cdots & Q_{n,0,m-1}(\beta)\\
\vdots & \vdots & \ddots & \vdots\\
Q_{n,m-2,0}(\beta)& Q_{n,m-2,1}(\beta) & \cdots & Q_{n,m-2,m-1}(\beta)
\end{pmatrix}\in {\rm{Mat}}_m(K).
$$ 
Our proof relies on a qualitative linear independence criterion \cite[Proposition~$5.6$]{DHK2} which is based on the method of C.~F.~Siegel ({\it{see}} \cite{Siegel1}).
Define real numbers:
\begin{align*}
&\mathbb{A}_{v_0}(\beta)=
\begin{cases}
\log|\beta|_{v_0}-\sum_{i=1}^m{\rm{h}}_{v_0}(\alpha_i)-m{\rm{h}}_{v_0}(\boldsymbol{\alpha})+m\sum_{i=1}^m\log \mu_{v_0}(s_i) & \ \text{if} \ v \nmid \infty,\\
\log|\beta|_{v_0}-\sum_{i=1}^m{\rm{h}}_{v_0}(\alpha_i)-m{\rm{h}}_{v_0}(\boldsymbol{\alpha})-m\log|2|_{v_0}  & \ \text{if} \ v \mid \infty,
\end{cases}\\
&U_{v_0}(\beta)=\left(\sum_{i=1}^m{\rm{h}}_v(\alpha_i)+(m-1)({\rm{h}}_v(\boldsymbol{\alpha})+{\rm{h}}_v(\beta))\right)+m\sum_{i=1}^m\log \mu_{v_0}(s_i).
\end{align*}

\medskip

We now restate Theorem~\ref{main} together with a linear independence measure. 
\begin{theorem} \label{main+measure}
We use the same notations in Theorem~\ref{main}. Let $v_0\in \mathfrak{M}_K$
such that $V_{v_0}(\beta)>0$.
Then the series $f_j(z)$ for $0 \leq j \le m-2$ converge around $\beta$ in $K_{v_0}$ and
for any positive number $\varepsilon$ with $\varepsilon<V_{v_0}(\beta)$, there exists an effectively computable positive number $H_0$ depending on $\varepsilon$ and the given data such that the following property holds.
For any ${{\boldsymbol{\lambda}}}=({{\lambda}},{{\lambda_{j}}})_{0\le j \le m-2} \in K^{m} \setminus \{ \bold{0} \}$ satisfying $H_0\le {\mathrm{H}}({{\boldsymbol{\lambda}}})$, then 
\begin{align*}
\left|{{\lambda}}+\sum_{j={0}}^{m-2}{{\lambda_{j}}}f_{j}(\beta)\right|_{v_0}>C(\beta,\varepsilon){\mathrm{H}}_{v_0}({{\boldsymbol{\lambda}}}) {\mathrm{H}}({{\boldsymbol{\lambda}}})^{-\mu(\beta,\varepsilon)}\enspace,
\end{align*}
where 
\begin{align*}
&\mu(\beta,\varepsilon)=\dfrac{\mathbb{A}_{v_0}(\beta)+{{U}}_{v_0}(\beta)}{V_{v_0}(\beta)-\epsilon},\\
&C(\beta,\varepsilon)=\exp\left(-{{\left(\frac{\log(2)}{V_{v_0}(\beta)-\varepsilon}+1\right)}}(\mathbb{A}_{v_0}(\beta)+{{U}}_{v_0}(\beta)\right).
\end{align*}
\end{theorem}
\begin{proof}
Firstly, we claim that $M_n$ is invertible by applying Theorem~\ref{nonvanish Delta}. 
To this end, we show that equations~\eqref{key assump2} and~\eqref{important assump} in Theorem~\ref{nonvanish Delta} hold for our $a(z), b(z)$.

\medskip

First, we claim that the assumptions \eqref{first} and \eqref{second} imply \eqref{key assump2}.
Suppose not. Then there exist a positive integer $n$ and a root $\alpha_i$ of $a(t)$ such that 
\[
na'(\alpha_i)+b(\alpha_i)=0.
\]
The assumption~\eqref{first} implies $a'(\alpha_i)\neq0$, and the above equality yields
\[
\dfrac{b(\alpha_i)}{a'(\alpha_i)}=-n\in \Z_{\le -1}.
\]
This contradicts the assumption~\eqref{second}.
Moreover, the assumption~\eqref{third} implies that equation~\eqref{important assump} holds.
Therefore, Theorem~\ref{nonvanish Delta} ensures that $\det M_n\in K\setminus\{0\}$.

\medskip 

For $v\in \mathfrak{M}_K$, we define functions $F_v:\N\longrightarrow \R_{\ge0}$ by
\begin{align*}
F_v(n)&=n\left(\sum_{i=1}^m{\rm{h}}_v(\alpha_i)+(m-1)({\rm{h}}_v(\boldsymbol{\alpha})+{\rm{h}}_v(\beta))\right)\\
&+
\begin{cases}
m\sum_{i=1}^m\log |\mu_n(s_i)|_v^{-1}+\log|d_{(m-1)(n+1)}(b_{m-1})|_v+o(n) & \ \text{if} \ v \nmid \infty,\\
m \log |4|_v+o(n)  & \ \text{if} \ v \mid \infty,
\end{cases}
\end{align*}
where $o(n)=0$ for all but finitely many finite places non-Archimedean places $v$. Notice that
\[
\limsup_{n\to \infty}\dfrac{1}{n}F_{v_0}(n)=U_{v_0}(\beta).
\]

Proposition~$\ref{estimate PQ}$ allows us to get
\begin{align*} 
\max_{\substack{0\le j \le m-2 \\ 0\le \ell \le m-1}}\log\, \max\{|P_{n,j}(\beta)|_v, |Q_{n,j,\ell}(\beta)|_{v}\}\le F_{v}(n).
\end{align*}

\medskip

Then Proposition~$\ref{Remainder}$ yields 
\begin{align*}
\max_{\substack{0 \le j \le m-2 \\ 0\le \ell \le m-1}} \log\, |\mathfrak{R}_{n,j,\ell}(\beta)|_{v_0}\le -\mathbb{A}_{v_0}(\beta)n+o(n).
\end{align*}

Lemma~\ref{well-known}~$(iii)$ yields
\begin{align*}
&\limsup_{n\to \infty}\dfrac{1}{n}\log d_{(m-1)(n+1)}(b_{m-1})=(m-1)\dfrac{{\rm{den}}(b_{m-1})}{\varphi({\rm{den}}(b_{m-1}))}\sum_{\substack{ 1\le j\le {{\rm{den}}(b_{m-1})} \\ (j,{\rm{den}}(b_{m-1}))=1}}\dfrac{1}{j},
\end{align*} 
and we obtain
$$\mathbb{A}_{v_0}(\beta)-\limsup_n \dfrac{1}{n}\sum_{v\neq v_0}F_v(n)\le V_{v_0}(\beta),$$
where $V_{v_0}(\beta)$ is the real number defined in \eqref{V}.
Using a qualitative linear independence criterion in \cite[Proposition~$5.6$]{DHK2} for 
$$ 
\vartheta_{j}=f_j(\beta) \ \ \text{for} \ \ 0\le j \le m-2,
$$ 
and the family of invertible matrices $(M_n)_n$, and applying above estimates, we obtain Theorem~$\ref{main}$.
\end{proof}

\section{Appendix: Jordan-Pochhammer equation} \label{Appendix}
The equation of the following form is called the Jordan-Pochhammer equation ({\it confer} \cite[$18.4$] {I}):
\begin{align} \label{J-P}
\sum_{i=0}^m\binom{-\mu}{i}Q^{(i)}(z)\dfrac{d^{m-i}}{dz^{m-i}}-\sum_{j=0}^{m-1}\binom{-\mu-1}{j}R^{(j)}(z)\dfrac{d^{m-1-j}}{dz^{m-1-j}}
\end{align}
where $\mu$ is a complex number, $\binom{\mu}{0}=1$ and
\begin{align*}
&\binom{\mu}{i}=\dfrac{\mu(\mu-1)\ldots(\mu-i+1)}{i!} \ \ \text{for} \ \ i\ge1,\\
&Q(z)=(z-\alpha_1)\ldots(z-\alpha_m)\in \C[z] \ \ \text{with disinct roots}, \\ %\ \ \alpha_i\neq \alpha_j \ \ (i\neq j),\\ 
&R(z)\in \C[z] \ \text{of degree at most} \ m-1.
\end{align*}
We observe that the equation~\eqref{J-P} is of Fuchsian type with singularities $\alpha_1,\ldots,\alpha_m,\infty$, and has the following Riemann scheme:
$$
\begin{Bmatrix}
  \alpha_1 & \cdots & \alpha_m &  \infty \\
  0            & \cdots &   0        &    -\mu-1\\
  1            & \cdots &   1        &    -\mu-2\\
  \vdots     & \ddots &   \vdots        &   \vdots\\
 m-2          & \cdots &   m-2   &    -\mu-m+1\\
 m+\mu+s_1-1 & \cdots &   m+\mu+s_m-1   &  -\mu-\gamma
\end{Bmatrix},
$$
where $s_i:=R(\alpha_i)/Q'(\alpha_i)$ for $1\le i \le m$ and $\gamma=s_1+\ldots+s_m$.
\begin{example}\label{reduce JP}
Let $m\ge 2$ be an integer and $a(z),b(z)\in \C[z]$, where $a(z)$ is a monic polynomial of degree $m$ with distinct roots, and $b(z)$ is a polynomial of degree at most $m-1$.
Denote the differential operator $$L=\dfrac{d^{m-1}}{dz^{m-1}}\left(a(z)\dfrac{d}{dz}-b(z)\right).$$
Then, the following identity holds: 
\begin{align*}
L&=\sum_{i=0}^m\binom{m}{i}a^{(i)}(z)\dfrac{d^{m-i}}{dz^{m-i}}-\sum_{j=0}^{m-1}\binom{m-1}{j}(a'(z)+b(z))^{(j)}\dfrac{d^{m-1-j}}{dz^{m-1-j}}.
\end{align*}
Thus, $L$ takes a form of a Jordan-Pochhammer equation with $Q(z)=a(z), R(z)=a'(z)+b(z)$ and $\mu=-m$.
\end{example}

\noindent
{\bf Acknowledgements.}

%The authors extend sincere gratitude to Professors Daniel Bertrand, Sinnou David, and Noriko Hirata-Kohno for their invaluable suggestions.
The author expresses sincere gratitude to the anonymous referees for their valuable comments. Special thanks are due to Professor Ga\"{e}l R\'{e}mond for his insightful suggestions that improved the presentation of this paper, and to Akihito Ebisu for helpful discussions regarding the Jordan-Pochhammer equation. This work was supported in part by the Research Institute for Mathematical Sciences, an international joint usage and research center located at Kyoto University. The author was also supported by JSPS KAKENHI Grant Number JP24K16905.

\bibliography{}

\medskip\vglue5pt
\vskip 0pt plus 1fill
\hbox{\vbox{\hbox{Makoto \textsc{Kawashima}}
\hbox{{\tt kawasima@mi.meijigakuin.ac.jp}}
\hbox{Institute for Mathematical Informatics}
\hbox{Meiji Gakuin University}
\hbox{Totsuka, Yokohama, Kanagawa}
\hbox{224-8539, Japan}
}}
\end{document}